\documentclass[10pt]{article}
\usepackage{graphicx} 
\usepackage{geometry, hyperref}
\usepackage{amsmath, amsthm, amssymb, amsfonts, mathtools, mathrsfs, array, enumitem}
\usepackage{tikz}
\usetikzlibrary{positioning}
\usetikzlibrary{decorations.pathreplacing}

\newcommand{\N}{\mathbb{N}}

\newcommand{\R}{\mathbb{R}}

\newcommand{\LL}{\mathcal{L}}

\newcommand{\Norm}[1]{\left \lVert #1 \right \rVert}
\newcommand{\Abs}[1]{\left \lvert #1 \right \rvert}

\newcommand{\Ceil}[1]{\left \lceil #1 \right \rceil}

\newcommand{\cali}[1]{\mathcal{#1}}
\newcommand{\Normm}{\left \lVert \; \cdot \; \right \rVert}
\newcommand{\set}[1]{\left \{ #1 \right \}}
\newcommand{\st}{\; \middle | \;}
\newcommand{\txt}[1]{\text{#1}}
\newcommand{\lr}[3]{\left #1 #2 \right #3}
\newcommand{\nW}[1]{\nw \left ( #1 \right )}
\newcommand{\op}[1]{\operatorname{#1}}

\newcommand{\func}[5]{
\setlength\arraycolsep{0pt}
#1\colon \begin{array}[t]{ >{\displaystyle}r >{{}}c<{{}}  >{\displaystyle}l } 
          #2 &\longrightarrow & #3 \\ 
          #4 &\longmapsto& #5 
         \end{array}
}

\newtheorem{lemma}{Lemma}[section]
\newtheorem{proposition}[lemma]{Proposition}
\newtheorem{theorem}[lemma]{Theorem}
\newtheorem{corollary}[lemma]{Corollary}

\newtheorem*{theorem**}{Theorem\theoremnum}
\newenvironment{theorem*}[1][]{
  \edef\theoremnum{\if\relax\detokenize{#1}\relax\else~#1\fi}
  \begin{theorem**}
}{
  \end{theorem**}
}

\DeclareMathOperator{\dimin}{dim_{in}}
\DeclareMathOperator{\dimout}{dim_{out}}
\DeclareMathOperator{\ReLU}{ReLU}
\DeclareMathOperator{\nw}{n_W}
\DeclareMathOperator{\Eye}{I}

\theoremstyle{definition}
\newtheorem{remark}[lemma]{Remark}

\newtheorem{hyp}{Hypothesis}
\newtheorem{deff}[lemma]{Definition}
\numberwithin{equation}{section}

\newtheorem*{hyp**}{Hypothesis\theoremnum}
\newenvironment{hyp*}[1][]{
  \edef\theoremnum{\if\relax\detokenize{#1}\relax\else~#1\fi}
  \begin{hyp**}
}{
  \end{hyp**}
}

\title{A theoretical analysis on the resolution of parametric PDEs via Neural Networks designed with Strassen algorithm}

\author{Gonzalo Romera\thanks{Department of Mathematics, University of the Basque Country UPV/EHU, Barrio Sarriena s/n, 48940, Leioa, Spain. E-mail: {\tt gonzalo.romera@ehu.eus}.} \and Jon Asier Bárcena-Pétisco\thanks{\textbf{Contact author.} [1] Department of Mathematics, University of the Basque Country UPV/EHU, Barrio Sarriena s/n, 48940, Leioa, Spain. E-mail: {\tt jonasier.barcena@ehu.eus}. Phone number +34 946015366. [2] BCAM - Basque Center for Applied Mathematics, Spain. E-mail: {\tt jabarcena@bcamath.org}}}

\begin{document}

\maketitle

\noindent
\textbf{Abstract:}

\noindent
We construct a family of Neural Networks that approximate matrix multiplication operator for any activation function such that there exists a Neural Network which can approximate the scalar multiplication function. In particular, we use the Strassen algorithm to bound the number of weights and layers needed for such Neural Networks. This allows us to define another Neural Network for approximating the inverse matrix operator. Finally, we discuss how it can be applied to numerically solve elliptic PDEs.

\vspace{.5cm}

\vspace{.5cm}
\noindent
\textbf{Abbreviated title:} Neural Networks and Strassen algorithm.

\vspace{.5cm}

\vspace{.5cm}
\noindent
\textbf{Acknowledgement:} Both authors were supported by the Grant PID2023-146764NB-I00 funded by MICIU/AEI/10.13039/501100011033 and cofunded by the European Union. Also, both authors were supported by the grant~IT1875-26 funded by the Basque Government.

\newpage

\section{Introduction}

Machine Learning models have flourished because of their applications in many fields. However, theoretical guaranties have not kept up with this speed. One of the open problems is related to the expressivity of Neural Networks: how many parameters are really needed to solve a given problem? The computational complexity of a Neural Network is completely determined by its number of parameters and activation function. Two relevant examples are the $L^p$ approximation of functions and numerically solving PDEs. In these particular cases, Neural Networks are a reasonable approach for such tasks (specially considering empirical results), but theoretical results usually lack upper bounds on the number of parameters.

Linked to these problems, we focus on this paper on the approximation of matrix multiplication and inversion operators, which are required, among other contexts, in the numerical resolution of PDEs these are essential operations that can then be replaced by a Neural Network.
In this paper, we employ the Strassen algorithm for the matrix product to show that the dependency on the number on the dimension of the matrix can be diminished, and in the use of power series to obtain the inverse of such matrices. This in return gives better upper bounds on the number of weights required in the numerical resolution of parametric elliptic PDEs with Neural Networks.

\subsection{Our Contribution}

Motivated by the paper \cite{kutyniok2022theoretical}, we seek to optimize and generalize Neural Networks capable of approximating the matrix-to-matrix product operator as well as the matrix inversion operator. Our contributions are the following ones:

\begin{itemize}
    \item \textbf{Activation functions.} In Neural Networks, the activation function plays a major role, as it determines its expressivity. In this paper, we focus on activation functions for which there are Neural Networks that approximate the product operation between two real numbers, as mathematically stated in Hypothesis \ref{hyp: product varrho realization approximated}. Clearly, if we cannot approximate such operation, there is no hope of representing the product of matrices, which is then needed to invert matrices. This family of maps contains $\ReLU$, the activation function considered in \cite{kutyniok2022theoretical}.

    \item \textbf{Bounds.} The authors of \cite{kutyniok2022theoretical} bounded the number of parameters needed for such Neural Networks. To achieve more precise bounds, we use two key ideas. First, we use Strassen algorithm for matrix multiplication, which reduces the number of scalar multiplications needed for matrix multiplication. This allows us to economize the use of the potentially bulky Neural Network that approximates the product of scalars. Secondly, we use a finite sum of powers of matrices to approximate the inverses of matrices. Rewriting that sum as in \eqref{eq: Neumann sum as product} below, we reduce the number of matrix multiplications in comparison to \cite{kutyniok2022theoretical}. 
    
\begin{itemize}
    \item  Corollary \ref{cl: MNN Strassen for square matrices} provides a Neural Network that approximates the matrix multiplication operator. We achieve a smaller size of Neural Network than the one in \cite{kutyniok2022theoretical} with a dependency on the number of rows of the input matrices that is no longer cubic, but of the order of $\log_2 7$. 

    \item   Theorem \ref{th: INV MNN} provides a Neural Network that approximates the matrix inverse operator. Again, we achieve a smaller size of Neural Network than the one in \cite{kutyniok2022theoretical} with a dependency on the number of rows of the input matrices that is no longer cubic, but of the order of $\log_2 7$. 

    \item As in \cite{kutyniok2022theoretical}, in Section \ref{sec: Applications}, we see how Theorem \ref{th: INV MNN} is used to numerically solve PDEs with the Galerkin method. The most common application of the Galerkin method, the finite element method, suffers from the \emph{curse of dimensionality}, which we diminish with the exponent $\log_2 7$. 
\end{itemize}

    \item \textbf{Notation.} Finally, we simplify notation by allowing matrices as input of Neural Networks. This avoids the vectorization of matrices, which overloads the notation.
\end{itemize}

\subsection{State of the art}

In this subsection, we discuss the results of the literature related to this paper. The most related work is \cite{kutyniok2022theoretical} since our paper is a generalization and improvement of its results, as we have explained before.

For the rest of the literature review, in Section \ref{subsubsec: Strassen algorithms} we explore the matrix multiplication algorithms which optimize the amount of scalar multiplications. Then, in Section \ref{subsubsec: NN and multiplication} we explore the different interactions involving matrix multiplication, Neural Networks and the Strassen algorithm. We follow in Section \ref{subsubsec: approximation NN} by presenting the approximation theory for Neural Networks, which justifies that the product of scalars can be approximated by Neural Networks with a great variety of activation functions.  Finally, in Section \ref{subsubsec: NN and PDEs} we present the uses that Neural Networks receive to numerically solve PDEs.

\subsubsection{Strassen-Like Algorithms} \label{subsubsec: Strassen algorithms}


In \cite{strassen1969gaussian}, it is shown that for matrices of size $n \times n$, only $\cali{O}(n^{\log_2 7})$ scalar multiplications are required, instead of $\theta(n^{3})$, which is what the definition suggests.

It is an open problem to find the infimum of $\omega$ such that for every $\varepsilon > 0$ there is a $(n \times n) $-matrix multiplication algorithm with $\cali{O}(n^{\omega + \varepsilon})$ arithmetic operations. For that purpose a number of more optimal algorithms than Strassen appeared. We would like to highlight two of them: \cite{COPPERSMITH1990251}, which describes an algorithm that was not outperformed for 20 years, and \cite{alman2024asymmetryyieldsfastermatrix}, which is the latest and lowest upper bound. Both use the laser method introduced by Strassen in \cite{STRASSEN+1987+406+443} that relates trilinear polynomials to the matrix multiplication tensor. This method, however, has a limitation: \cite{ambainis2015fast} shows that using that method and a variety of its generalizations cannot obtain the bound for $\omega = 2$, which is a reasonable target.

Even if those papers present better algorithms than the one in \cite{strassen1969gaussian}, as the structure of the product is much more complicated or unknown, we leave it for future work.

\subsubsection{Neural Networks and Multiplication} \label{subsubsec: NN and multiplication}

 
For scalar multiplication, in \cite{petersen2018optimal, yarotsky2017error}  a $\ReLU$ Neural Network capable of approximating the product of any pair of scalars in an interval is described. Then, using that Neural Network as building block, the authors of \cite{kutyniok2022theoretical} construct Neural Networks with $\ReLU$ as activation function which approximates the matrix multiplication and inverse operators. This is, as far as we know, the first explicitly constructed $\ReLU$ Neural Network to approximate matrix multiplication or its inverse with upper bounds on the number of weights and layers. For other activation functions like $\varrho(x) = x^2$ and $\ReLU^2$, those results are obvious while others, even if it is known to exist, there are upper bounds on the number of parameters for a Neural Network that approximates scalar multiplication operator which are non sharp and general bounds given by approximation theorems, like in \cite{mhaskar1996neural}.


On the reciprocal problem of this work, Strassen-type matrix multiplication algorithms have been obtained using Neural Networks, by \cite{elser2016network, fawzi2022discovering} among others. Additional examples of those algorithms can be found in the review \cite{mansour2024review}. However, it should be noted that the algorithms mentioned in Section \ref{subsubsec: Strassen algorithms} are more efficient asymptotically. 


Optimization of evaluations and training of Neural Networks and other related architectures of Machine Learning is a huge and ever-evolving research topic. Among many other algorithms, the Strassen-type matrix multiplication algorithms are used in this context. This has been done from a purely computational perspective, for example, in \cite{ali2020reduction, rao2018winograd,zhao2018faster}. We are unaware of theoretical analysis of this kind of algorithms involving Strassen-type matrix multiplications.

\subsubsection{Approximation Theory for Neural Networks} \label{subsubsec: approximation NN}

Approximation theory is critical to justify the success of Neural Networks. It studies its density and convergence to the desired function with the objective of answering when a Neural Network is suited for a given job.


The first results in approximation theory for Neural Networks are density theorems over $C^0$ and $L^p$ spaces, results known as Universal Approximation Theorems. The seminal work is \cite{cybenko1989approximation}, which proved the density of one layer Neural Networks with sigmoidal activation functions in the space of continuous functions equipped with the supremum norm. Generalizations followed, such as \cite{hornik1989multilayer}, where a more general set of maps containing Neural Networks is proved to be dense in $C^0$, and such as \cite{leshno1993multilayer}, where the authors characterize for what activation functions the set of Neural Networks is dense in $L^p$ with $p \in (1, \infty)$. Density results are not limited to classic Lebesgue spaces. In fact, the author of \cite{hornik1991approximation} gives conditions on the activation function ($m$ times continuously derivable and bounded) for the density in Sobolev spaces over compact and finite measures. As for Besov spaces, its relations with the approximation classes of Neural Network is reflected in \cite{gribonval2022approximation}. However, in these papers, there is no computation on the number of weights required to approximate a given function.


In order to give upper bounds for the width and depth of Neural Network, some smoothness of the function is usually required. Examples of this are given in the papers \cite{barron1993universal, mhaskar1996neural, yarotsky2017error, petersen2018optimal, hanin2019universal, guhring2020error, opschoor2020deep, guhring2021approximation, siegel2023optimal}. The pioneer is \cite{barron1993universal}, which gives a relation between the number of weights of a one layer Neural Network and its proximity in $L^2$ norm to functions with an integrability condition on its Fourier Transform. For \cite{mhaskar1996neural,siegel2023optimal,yarotsky2017error}, the $L^p$ proximity of functions in $W^{s, p}$ with Neural Networks is expressed in terms of the number of weights and layers, and the Sobolev norm. In \cite{opschoor2020deep}, the authors do the same for ReLU Neural Networks but for the Besov spaces and it is a sharp bound. For $L^p$ norms, the authors of \cite{petersen2018optimal} show the proximity of continuously differentiable to Neural Networks while tracking the necessary number of parameters. Moreover, given a continuous function, in \cite{han2017deep} it is shown that there exists a sequence of Neural Networks with known number of parameters converging in the supreme norm to that function. In addition, in \cite{guhring2020error, guhring2021approximation}, the authors provide upper bounds on the number of weights and layers to approximate in the Sobolev space $W^{s, p}$ by Neural Networks a target function which belongs to $W^{n, p}$, for $n > s$. Finally, the survey \cite{Guhring_Raslan_Kutyniok_2022} provides a general overview of this topic.


From a more computational perspective, the paper \cite{du2019gradient} proves that the most widely used algorithm for training Neural Networks, gradient descent, achieves zero $\ell^2$ regression error as long as the activation function is smooth and Lipschitz. On the same topic, \cite{allen2019convergence} guarantees, for the same algorithm and for sufficiently large Neural Networks and more general loss functions, convergence in the same sense, as well as providing estimates on its time cost.


Regarding the approximation of operators rather than functions, in recent years two important frameworks have appear: \textit{DeepONet} and \textit{Neural Operators}. The first one was introduced in \cite{lu2019deeponet} as an extension of the seminar work \cite{chen1995universal}. Both prove an universal approximation theorem for continuous operators. One year after its introduction, the results in \cite{lanthaler2022errorestimatesdeeponetsdeep} extends the results to $L^2$ operators. For similar structures, continuous operators between separable Hilbert spaces are point-wise approximated by Operator nets, as proved in \cite{schwab2023deep} and in \cite{herrmann2024neural} with convergence rates in a particular subset. Very recently, the study of set-valued maps was carried out in \cite{garciauniversal}.
For Neural Operators, their first appearance is attributed to the paper \cite{li2020neural}, where empirical tests are performed to solve some PDEs. In order to reduce the computational cost of evaluating these architectures, some of these authors consider a particular implementation in \cite{li2020fourier} where the Fourier Transformed is applied, giving us \textit{Fourier Neural Operators}. This last, according to \cite{kovachki2021universal}, satisfy the universal approximation properties for continuous maps between Sobolev spaces in the torus.

\subsubsection{Neural Networks and PDEs} \label{subsubsec: NN and PDEs}


Machine Learning techniques for numerically solving PDEs are a growing research field. The most relevant of this trend is the Physics Informed Neural Networks introduced in \cite{raissi2019physics,sirignano2018dgm}. A Physics Informed Neural Network is a Neural Network that approximates the solution of a PDE by minimizing a loss function derived from the governing PDE. Other examples of using Neural Networks that are not Physics Informed Neural Networks to approximate numerically solutions of specific PDEs are in \cite{han2017deep, hutzenthaler2020proof}.

More recently, the aim of architectures has shifted from approximating particular solutions of a PDE to approximating its solver operator. As the experiments suggest in \cite{lu2019deeponet} and \cite{li2020neural}, DeepONets and Neural Operators respectively learn the solution maps of some PDEs. Also, in \cite{kovachki2021universal} Fourier Neural Operators are used to numerically solve some PDE problems. Concerning DeepONets, in \cite{lanthaler2022errorestimatesdeeponetsdeep} a number of PDEs are listed in which the curse of dimensionality is avoided by using DeepONets. 


There is also a clash between the classical algorithms for numerically solving PDEs with this new wave of Machine Learning based algorithms. This results in mixing the two approaches to mitigate the computational costs of both techniques. Such classical algorithms are the Galerkin method and the Finite Element Method. Combining those methods, which require an appropriate selection of functions, with Neural Networks, we get results such as those in \cite{he2018relu, lee2020model, dong2021local, zhang2021hierarchical, geist2021numerical}. The idea is to replace the hard-to-find functions that form a basis by a trainable family of Neural Networks. 
Other ideas are exploited in \cite{kutyniok2022theoretical} with the aim of providing a mathematical understanding instead of computational efficiency. Using related ideas, in \cite{JMLR:v24:23-0421}, they discretize the space of solutions of stationary linear diffusion equation by means of Finite Elements, and, supported by algorithms for solving the resulting linear system, they develop multilevel Neural Networks to approximate the solutions of those PDEs. Finally, mixing the two approaches is not limited either to Neural Networks. For example, the authors of \cite{gao2022physics} propose Graph Convolutional Networks for this task.

\subsection{Outline}

The outline of the paper is the following: in Section \ref{sec: Preliminaries}, we fix the notation, describe the Strassen Algorithm, and define Matrix Neural Networks and related concepts. In Section \ref{sec: Strassen Neural Networks}, we construct Matrix Neural Networks that approximate the matrix multiplication operator. In Section \ref{sec: Strassen Matrix Multiplication for Inverting Matrices}, we construct a Matrix Neural Network that approximates the inverse of a set of matrices. Finally, in Section \ref{sec: Applications} we explain the use of the previous Matrix Neural Network for numerically solving PDEs.

\section{Preliminaries} \label{sec: Preliminaries}

\subsection{Linear Algebra}

In this paper, we denote the spaces of matrices of $n_1$ rows and $n_2$ columns by $\R^{n_1 \times n_2}$. Given $A \in \R^{n_1 \times n_2}$, $i \in \set{1, \dots, n_1}$ and $j \in \set{1, \dots, n_2}$, $A_{i,j}$ denotes the entry in its $i$-th row and $j$-th column. Moreover, given $c_1, c_2 \in \set{1, \dots, n_1}$ such that $c_1 \leq c_2$ and $c_3, c_4 \in \set{1, \dots, n_2}$ such that $c_3 \leq c_4$ we denote by $A_{c_1:c_2, c_3: c_4}$ its submatrix formed by the rows from $c_1$ to $c_2$ (including $c_1$ and $c_2$), and the columns from $c_3$ to $c_4$ (including $c_3$ and $c_4$). We also denote by $\Eye_n \in \R^{n \times n}$ the identity matrix whose entries are $(\Eye_n)_{i, j} = \delta_{i, j}$ for all $i, j \in \set{1, \dots, n}$. 

Given two matrices $A \in \R^{n_1 \times n_2}$ and $B \in \R^{n_3 \times n_4}$ we write $A|B \in \R^{n_1 \times(n_2 + n_4)}$ when $n_1 = n_3$ for the matrices whose columns are those of $A$ and $B$ in that order. Similarly, for rows, we use $\begin{pmatrix} A \\ B \end{pmatrix}$ whenever $n_2 = n_4$.

In addition, we consider the following notation that is especially useful for describing the Strassen algorithm:

\begin{deff} \label{def: matrix split in four}
    Let $A \in \R^{2^{k + 1} \times 2^{k + 1}}$ for some $k \in \N \cup \set{0}$. Then, we denote:
    \begin{align*}
        A[1, 1] &\coloneq A_{1: 2^k, 1: 2^k}, & A[1, 2] &\coloneq A_{1: 2^k, 2^k + 1: 2^{k + 1}}, \\
        A[2, 1] &\coloneq A_{2^k + 1: 2^{k + 1}, 1: 2^k}, & A[2, 2] &\coloneq A_{2^k + 1: 2^{k + 1}, 2^k + 1: 2^{k + 1}}.
    \end{align*}
    Note that we have:
    \[
    A = \begin{pmatrix}
        A[1, 1] | A[1, 2] \\
        A[2, 1] | A[2, 2]
    \end{pmatrix}.
    \]
\end{deff}

Throughout this paper, we consider linear functionals between spaces of matrices. In particular, we consider tensors located in $\R^{(n_1 \times n_2) \times (n_3 \times n_4)}$. We recall that, given $\LL \in \R^{(n_1 \times n_2) \times (n_3 \times n_4)}$ and $A \in \R^{n_3 \times n_4}$, the product $\LL A$ is the matrix in $\R^{n_1 \times n_2}$ given by:
\[
(\LL A)_{i,j} = \sum_{l= 1}^{n_4} \sum_{k = 1}^{n_3} \LL_{i, j, k, l} A_{k,l}.
\]
In particular, when $n_1\times n_2=n_3\times n_4$, we define the identity tensor $\op{id}$, whose entries are:
\[(\op{id})_{i,j}=\delta_{i\times j,k\times l}.\]
\paragraph{}
An important notion when working with Neural Network is the notion of weights; that is, of non null terms. For this purpose, we define  for all $A \in \R^{n_1 \times n_2}$ its number of weights:
\[
\nW{A} = \# \set{(i, j) \in \set{1, \dots, n_1} \times \set{1, \dots, n_2} \st A_{i, j} \neq 0}.
\]
Similarly, we extend the notion to vectors and tensors.

Finally, we denote by $\Abs{\;\cdot\;}_p$ the usual $p$ norm for vectors and matrices (seen as vectors) and by $\Normm_p$ the matrix norm induced by $\Abs{\;\cdot\;}_p$.

In particular, we need the following classical estimates between  norms:
\begin{lemma}
Let $A\in \R^{n \times n}$. Then,
   \begin{equation} \label{eq: infinity and 2 norm equivalence constants}
        \Abs{A}_\infty \leq \Norm{A}_2 \leq n \Abs{A}_\infty.
    \end{equation}
\end{lemma}

\subsection{Strassen Algorithm for Square Matrices} \label{subsec: strassen for any square matrix}

In this section, we recall the Strassen algorithm, which was proposed in \cite{strassen1969gaussian}. First, we define the Strassen algorithm in matrices belonging to $\R^{2^k \times 2^k}$. Then, we define it in matrices in $\R^{n \times n}$ for arbitrary values $n \in \N$.

Let us now recursively define the Strassen algorithm for the matrices in $\R^{2^k \times 2^k}$. The base case $k = 0$ simply consists of multiplications of scalars. The inductive case, $k\geq1$, is defined as follows:
\begin{equation} \label{align: product of matrices as seven products}
\begin{aligned}
    (AB)[1, 1] &= P_1 + P_4 - P_5 + P_7, & (AB)[1, 2] &= P_3 + P_5, \\
    (AB)[2, 1] &= P_2 + P_4, & (AB)[2, 2] &= P_1  - P_2 + P_3 + P_6,
\end{aligned}
\end{equation}
for:
\begin{equation} \label{align: Strassen products}
\begin{split}
    P_1 &= (A[1, 1] + A[2, 2]) (B[1, 1] + B[2, 2]), \\
    P_2 &= (A[2, 1] + A[2, 2]) B[1, 1], \\
    P_3 &= A[1, 1] (B[1, 2] - B[2, 2]), \\
    P_4 &= A[2, 2] (B[2, 1] - B[1, 1]), \\
    P_5 &= (A[1, 1] + A[1, 2]) B[2, 2], \\
    P_6 &= (A[2, 1] - A[1, 1]) (B[1, 1] + B[1, 2]), \\
    P_7 &= (A[1, 2] - A[2, 2]) (B[2, 1] + B[2, 2]),
\end{split}
\end{equation}
where all the multiplications in \eqref{align: Strassen products} consist of the Strassen multiplication in $\mathbb R^{2^{k-1}\times 2^{k-1}}$.

When $n \neq 2^k$, the Strassen algorithm can be adapted in different ways:
\begin{enumerate}
    \item We extend the matrix from $n \times n$ to $2^k \times 2^k$, for $k = \Ceil{\log_2(n)}$ as follows:
    \[
    A \mapsto \begin{bmatrix}
        A & 0 \\ 0 & 0
    \end{bmatrix}.
    \]
    It is then enough to apply the Strassen multiplication to this extended matrix and recover the convenient submatrix.
    
    \item We consider a division of blocks of size $2^k \times 2^k$, perform the multiplication with such blocks, and do the corresponding additions. For example, if $n = 6$ and $A, B \in \R^{6 \times 6}$,  we can express the product $AB$ as a sum of products of submatrices where we can apply the Strassen algorithm:
    \begin{align*}
        (AB)_{1: 4, 1: 4} & = A_{1: 4, 1: 4} B_{1: 4, 1: 4} + \begin{pmatrix} A_{1: 2, 5: 6} B_{5: 6, 1: 2} & A_{1: 2, 5: 6} B_{5: 6, 3: 4} \\ A_{3: 4, 5: 6} B_{5: 6, 1: 2} & A_{3:4, 5: 6} B_{5: 6, 3: 4} \end{pmatrix}, \\
        (AB)_{1: 2, 5: 6} & = A_{1: 2, 1: 2} B_{1: 2, 5: 6} + A_{1: 2, 3: 4} B_{3: 4, 5: 6} + A_{1: 2, 5: 6} B_{5: 6, 5: 6}, \\
        (AB)_{3: 4, 5: 6} & = A_{3: 4, 1: 2} B_{1: 2, 5: 6} + A_{3: 4, 3: 4} B_{3: 4, 5: 6} + A_{3: 4, 5: 6} B_{5: 6, 5: 6}, \\
        (AB)_{5: 6, 1: 2} & = A_{5: 6, 1: 2} B_{1: 2, 1: 2} + A_{5: 6, 3: 4} B_{3: 4, 1: 2} + A_{5: 6, 5: 6} B_{5: 6, 1: 2}, \\
        (AB)_{5: 6, 3: 4} & = A_{5: 6, 1: 2} B_{1: 2, 3: 4} + A_{5: 6, 3: 4} B_{3: 4, 3: 4} + A_{5: 6, 5: 6} B_{5: 6, 3: 4}, \\
        (AB)_{5: 6, 5: 6} & = A_{5: 6, 1: 2} B_{1: 2, 5: 6} + A_{5: 6, 3: 4} B_{3: 4, 5: 6} + A_{5: 6, 5: 6} B_{5: 6, 5: 6}.
    \end{align*}
\end{enumerate}

\subsection{Matrix Neural Networks}

We now define the notions of Matrix Neural Network, concatenation, and parallelization. This is based on the definition introduced in \cite{petersen2018optimal} and \cite{gribonval2022approximation}, but adapted for Neural Networks whose entries are matrices.
We do this for the sake of simplifying the notation.

\begin{deff} \label{def: MNN}
    Let $\varrho: \R \longrightarrow \R$. A $\varrho$-\emph{Matrix Neural Network} ($\varrho$-MNN from now on) is a tuple $((\LL^1, C^1, \alpha^1), \dots, (\LL^L, C^L, \alpha^L))$ where each $\LL^\ell$ is a $(N^{\ell, 1} \times N^{\ell, 2}) \times (N^{\ell - 1, 1} \times N^{\ell - 1, 2})$ tensor, $C^\ell$ is a $N^{\ell,1} \times N^{\ell, 2}$ matrix and
    \[
    \func{\alpha^\ell}{\R^{N^{\ell, 1} \times N^{\ell, 2}}}{\R^{N^{\ell, 1} \times N^{\ell, 2}}}{(A_{i, j})}{(\alpha^\ell_{i, j}(A_{i, j}))}
    \]
    satisfying $\alpha^\ell_{i, j} \in \set{\op{id}, \varrho}$ if $\ell < L$ and $\alpha^L = \op{id}$.
\end{deff}

\begin{remark}
Multilayer Perceptron (if we consider them defined as in \cite{petersen2018optimal}) have the same structure as a Matrix Neural Networks: they are a sequence of triples of matrices, vectors, and activation functions. Their realization, number of layers, and number of weights are defined equivalently for these Neural Networks. In particular, any Multilayer Perceptron can be viewed as a Matrix Neural Network by considering $N^{\ell, 2} = 1$ for all $\ell = 1, \dots, N$. Conversely, using vectorization procedures, since the matrix space $\R^{n \times m}$ is isomorphic to the vector space $\R^{nm}$, all the Matrix Neural Networks can be implemented as classical Neural Networks (with the same length, number of weights, and final error). Therefore, we can view both notions as equivalent.
\end{remark}

We can directly extend for a $\varrho$-MNN $\Phi$ the definitions given in \cite{kutyniok2022theoretical} of its \emph{input} and \emph{output dimensions} $\dimin(\Phi), \dimout(\Phi)$, \emph{number of layers} $L(\Phi)$, \emph{number of weights at the layer} $\ell$ $M_\ell(\Phi)$, \emph{number of weights} $M(\Phi)$ and \emph{realization} $R(\Phi)$. We omit their formal definitions for the sake of brevity as they are standard notions.

For more than one $\varrho$-MNN $\Phi^1, \dots, \Phi^k$, we can again extend the ideas of their \emph{sparse concatenation} $\Phi^1 \odot \Phi^2$ and their \emph{parallelization} $\cali{P}(\Phi^1, \dots, \Phi^k)$ as defined in, for example, \cite{petersen2018optimal} and \cite{kutyniok2022theoretical}. In these papers, they focused in the activation function $\ReLU$ and these definitions are written with that in mind, but, allowing components of the activation functions to be the identity solves this technicalities. The whole point of introducing these concepts is that they satisfy the following:
\[
R(\Phi^1 \odot \Phi^2) = R(\Phi^1) \circ R(\Phi^2)
\]
and
\[
R(\cali{P}(\Phi^1, \dots, \Phi^k)) \begin{pmatrix}
    A_1 \\ \vdots \\ A_k
\end{pmatrix} = \begin{pmatrix}
    R(\Phi_1)(A_1) \\ \vdots \\ R(\Phi_k)(A_k)
\end{pmatrix},
\]
while keeping track in its number of weights and layers.

In order to implement the Strassen algorithm with $\varrho$-MNN,  we suppose that we may multiply two numbers using Neural Networks:
\begin{hyp} \label{hyp: product varrho realization approximated}
    Let $\varrho$ be an activation function. Then, for all $K, \varepsilon > 0$ there is a $\varrho$-MNN $\Pi_{\varepsilon, K}^\varrho$ such that:
    \[
    \sup_{(x, y) \in [- K, K]^2} \Abs{xy - R(\Pi_{\varepsilon, K}^\varrho)(x, y)} \leq \varepsilon.
    \]
\end{hyp}

This is easily satisfied by the multiple versions of the Universal Approximation Theorem, for example, that in \cite{cybenko1989approximation}. Actually, Hypothesis \ref{hyp: product varrho realization approximated} is equivalent to these results: the product approximation unlocks the polynomial approximation and therefore, the approximation of continuous functions.

\begin{remark}
    An easy example is $\varrho = \ReLU^2$ where we can compute the product exactly. This is proven by just noticing that $\varrho(t) + \varrho(-t)=t^2$.
\end{remark}

\begin{remark}
    For the particular case $\varrho = \ReLU = t \mapsto\max\set{t,0}$, $\Pi_{\varepsilon, K}^{\ReLU}$ is constructed in \cite{elbrachter2022dnn}, which, by a minor modification using the equality
    \[
    4xy = (x+y)^2 - (x-y)^2
    \]
    instead of 
    \[
    2xy = (x+y)^2 -(x^2+y^2),
    \]
    we can slightly improve the upper bounds on the number of weights:
    \begin{equation}\label{item: product weights*}
        M \lr{(}{\Pi_{\varepsilon, K}^{\ReLU}}{)} \leq \begin{cases} 30 \log_2 K + 15 \log_2 \frac{1}{\varepsilon} + 25 & \txt{ if } \varepsilon < \frac{K^2}{2}, \\ 12 & \txt{ if } \varepsilon \in \left [ \frac{K^2}{2}, K^2 \right ), \\ 0 & \txt{ if } \varepsilon \geq K^2, \end{cases}
    \end{equation}
    and on the number of layers:
    \begin{equation}\label{item: product layers*}
        L\lr{(}{\Pi_{\varepsilon, K}^{\ReLU}}{)} \leq \begin{cases} \log_2 K + \frac{1}{2} \log_2 \frac{1}{\varepsilon} + \frac{5}{2} & \txt{ if } \varepsilon < \frac{K^2}{2}, \\
        2 & \txt{ if } \varepsilon \in \left [ \frac{K^2}{2}, K^2 \right ), \\ 1 & \txt{ if } \varepsilon \geq K^2. \end{cases}
    \end{equation}
    
\end{remark}


\section{Matrix Multiplication Approximation}\label{sec: Strassen Neural Networks}

We proceed to define the Strassen MNN to approximate the multiplication of matrices in $\R^{2^k \times 2^k}$ assuming Hypothesis \ref{hyp: product varrho realization approximated}. In particular, we consider Neural Networks that are defined recursively on $k$:

The base case is given by:
\begin{equation}\label{eq:Phibase}
R \left ( \Phi^{\op{STR}_{2^0, \varrho}}_{\varepsilon, K} \right )(A|B) = R \left ( \Pi_{\varepsilon, K}^{\varrho} \right )(A_{1, 1}, B_{1, 1}),
\end{equation}
for $\Pi_{\varepsilon, K}^{\varrho}$ defined in Hypothesis \ref{hyp: product varrho realization approximated}.

For the recursive case, we need to define several auxiliary networks. The first, whose input is in $\R^{(7 \cdot 2^{k - 1}) \times 2^{k - 1}}$ and whose output is in  $\R^{2^k \times 2^k}$, is given by:
\begin{equation*}
R\left(\Phi^{\op{MIX}_k}\right) \begin{bmatrix} P_1 \\ P_2 \\ P_3 \\ P_4 \\ P_5 \\ P_6 \\ P_7 \end{bmatrix} = \begin{pmatrix} P_1 + P_4 - P_5 + P_7 & P_3 + P_5 \\ P_2 + P_4 & P_1 - P_2 + P_3 + P_6 \end{pmatrix}.
\end{equation*}
Note that $\Phi^{\op{MIX}_k}$ is a one layer MNN such that:
\begin{equation} \label{eq: number of weights mix 1}
M\left (\Phi^{\op{MIX}_k} \right ) = 12 \cdot (2^{k - 1})^2 = 3 \cdot 4^k.
\end{equation}
Also, with the notation of Definition \ref{def: matrix split in four}, we define a MNN whose input is in $\R^{2^k \times 2^{k + 1}}$ and whose output is in $\R^{7 \cdot 2^{k - 1} \times 2 \cdot 2^{k - 1}}$:
\[
R(\Phi^{\op{SPLIT}_k})(A|B) = \begin{pmatrix} A[1, 1] + A[2, 2] & | & B[1, 1] + B[2, 2] \\ A[2, 1] + A[2, 2] & | & B[1, 1] \\ A[1, 1] & | & B[1, 2] - B[2, 2] \\ A[2, 2] & | & B[2, 1] - B[1, 1] \\ A[1, 1] + A[1, 2] & | & B[2, 2] \\ A[2, 1] - A[1, 1] & | & B[1, 1] + B[1, 2] \\ A[1, 2] - A[2, 2] & | & B[2, 1] + B[2 ,2] \end{pmatrix}.
\]
This network satisfies:
\begin{equation} \label{eq: number of weights SPLIT 1}
    M(\Phi^{\op{SPLIT}_k}) = 24 \cdot (2^{k - 1})^2 = 6 \cdot 4^k.
\end{equation}
Finally, we define a MNN whose input is in $\R^{7 \cdot 2^{k - 1} \times 2 \cdot 2^{k - 1}}$ and whose output is in $\R^{7 \cdot 2^{k - 1} \times 2^{k - 1}}$:
\begin{equation*}
    \Phi^{\op{PAR}_{k, \varrho}}_{\varepsilon, K} = \cali{P} \left (\Phi^{\op{STR}_{2^{k - 1}, \varrho}}_{\frac{\varepsilon}{4}, 2 K}, \Phi^{\op{STR}_{2^{k - 1}, \varrho}}_{\frac{\varepsilon}{4}, 2 K}, \Phi^{\op{STR}_{2^{k - 1}, \varrho}}_{\frac{\varepsilon}{4}, 2 K},  \Phi^{\op{STR}_{2^{k - 1}, \varrho}}_{\frac{\varepsilon}{4}, 2 K}, \Phi^{\op{STR}_{2^{k - 1}, \varrho}}_{\frac{\varepsilon}{4}, 2 K}, \Phi^{\op{STR}_{2^{k - 1}, \varrho}}_{\frac{\varepsilon}{4}, 2 K}, \Phi^{\op{STR}_{2^{k - 1}, \varrho}}_{\frac{\varepsilon}{4}, 2 K} \right ).
\end{equation*}
Bearing this in mind, we define recursively that:
\[
\Phi^{\op{STR}_{2^k, \varrho}}_{\varepsilon, K} = \Phi^{\op{MIX}_k} \odot \Phi^{\op{PAR}_{k, \varrho}}_{\varepsilon, K} \odot \Phi^{\op{SPLIT}_k}.
\]
Its properties regarding accuracy and number of parameters are summarized in the following proposition.

\begin{proposition} \label{prop: MNN STR powers of 2}
    Let $k \in \N \cup \set{0}$, $A, B \in \R^{2^k \times 2^k}$, $\varepsilon > 0$, and $K > 0$ satisfying:
    \[
    \Abs{A}_\infty, \Abs{B}_\infty \leq K.
    \]
    Then,
    \begin{align}
        & \Abs{R \left ( \Phi^{\op{STR}_{2^k, \varrho}}_{\varepsilon, K} \right ) (A|B) - AB}_\infty \leq \varepsilon, \label{align: MNN STR product} \\
        & M\left ( \Phi^{\op{STR}_{2^k, \varrho}}_{\varepsilon, K} \right ) = 7^k \left ( M \left ( \Pi^{\varrho}_{\frac{\varepsilon}{4^k}, 2^k K} \right ) + 12 \right ) - 12 \cdot 4^k, \label{align: MNN STR number of weights} \\
        & L \left ( \Phi^{\op{STR}_{2^k, \varrho}}_{\varepsilon, K} \right ) = L \left ( \Pi_{\frac{\varepsilon}{4^k}, 2^k K}^\varrho \right ) + 2 k \label{align: MNN STR number of layers}.
    \end{align}
\end{proposition}

The result is proved in Appendix \ref{sec: proof of MNN STR powers of 2}.

Let us now extend this result to matrices of any shape. Consider $m, n, p \in \N$. In order to multiply a matrix $A\in\R^{m \times n}$ by another matrix $B\in\R^{n \times p}$ using the Strassen algorithm implemented with Neural Networks, we are going to prolong the matrices to size $2^k \times 2^k$, for:
\[
k = k_{m, p}^n \coloneq \Ceil{\log_2 (\max\set{m, n, p})}.
\]
In particular, we recall the estimate:
\begin{equation} \label{eq: upper bound ceiling}
    \log_2 (\max \set{m, n, p})) <  k_{m, p}^n = \Ceil{\log_2 (\max \set{m, n, p}))} \leq \log_2 (\max \set{m, n, p}) + 1.
\end{equation}

For that purpose, we define a Neural Network whose input is in  $\R^{n \times (m + p)}$ and whose output is in $\R^{2^k \times 2 \cdot 2^k}$ with $k = k_{m, p}^n$:
\begin{equation*}
    R \left ( \Phi^{\op{EXT}_n}_{m, p} \right ) \left ( A^T \middle | B \right ) = 
    \left .
    \begin{pmatrix}
        A & 0^{m, 2^k - n} \\ 0^{2^k - m, n} & 0^{2^k - m, 2^k - n}
    \end{pmatrix} 
    \middle |
    \begin{pmatrix}
        B & 0^{n, 2^k - p} \\ 0^{2^k - n, p} & 0^{2^k - n, 2^k - p}
    \end{pmatrix}
    \right .  \in \R^{2^k \times 2 \cdot 2^k},
\end{equation*}
where $A^T$ denotes the transpose of $A$ and $0^{m, n} = (0)_{\substack{i = 1, \dots, m \\ j = 1, \dots, n}} \in \R^{m \times n}$. Transposing the left-hand side matrix is necessary to obtain a well-defined matrix as input.
Similarly, we define a Neural Network whose input is in $\R^{2^{k_{m, p}^n} \times 2^{k_{m, p}^n}}$
and output is in $\R^{m\times n}$:
\[
R \left ( \Phi_{m, p}^{\op{SHR}_n} \right )(A) = A_{1:m, 1:p} \in \R^{m \times p}.
\]
It is easy to check that:
\begin{align} \label{eq: number weights EXT and SHR}
    M \left ( \Phi^{\op{EXT}_n}_{m, p} \right ) = n  (m + p) && \txt{ and } && M \left ( \Phi_{m, p}^{\op{SHR}_n} \right ) = mp.
\end{align}
Thus, we can define the Neural Network:
\[
\Phi_{\varepsilon, K}^{\op{STR}_{m, p, \varrho}^n} = \Phi^{\op{SHR}_n}_{m, p} \odot \Phi_{\varepsilon, K}^{\op{STR}_{2^{k_{m, p}^n}, \varrho}} \odot \Phi^{\op{EXT}_n}_{m, p} .
\]

\begin{corollary} \label{cor: MNN STR rect properties}
    Let $A \in \R^{m \times n}$, $B \in \R^{n \times p}$ and $K > 0$ satisfying:
    \[
    \Abs{A}_\infty, \Abs{B}_\infty \leq K.
    \]
    Then, if we call $\gamma = \max\set{m,n, p}$, we have for every $\varepsilon > 0$ the following:
    \begin{align}
        & \Abs{R \left ( \Phi_{\varepsilon, K}^{\op{STR}_{m, p, \varrho}^n} \right ) \left ( A^T \middle | B \right ) - AB}_\infty \leq \varepsilon, \label{eq: MNN STR rect product} \\
        & M \left ( \Phi_{\varepsilon, K}^{\op{STR}_{m, p, \varrho}^n} \right ) \leq 7 \gamma^{\log_2 7} \left ( M \left ( \Pi^{\varrho}_{\frac{\varepsilon}{4 \gamma^2}, 2 \gamma K} \right ) + 12 \right )  - 9 \gamma^2, \label{eq: MNN STR rect number of weights} \\
        & L \left ( \Phi_{\varepsilon, K}^{\op{STR}_{m, p, \varrho}^n} \right ) \leq L \left ( \Pi_{\frac{\varepsilon}{4 \gamma^2}, 2 \gamma K}^\varrho \right ) + 2(\log_2\gamma + 2) \label{eq: MNN STR rect number of layers}.
    \end{align}
\end{corollary}

This result is a direct consequence of Proposition \ref{prop: MNN STR powers of 2} and the definition of $\Phi^{\op{STR}}$ and $\Phi^{\op{EXT}}$.

\begin{remark}
    When considering $\varrho = \ReLU$, we get, using \eqref{item: product weights*} and \eqref{item: product layers*}, the following asymptotic bounds:
    \begin{align}
        M \left ( \Phi^{\op{STR}_{m, p, \varrho}^n}_{\varepsilon, K} \right ) & \leq C \gamma^{\log_2 7} \left [ \log_2 \frac{1}{\varepsilon} + \log_2 \gamma + \log_2 \max \set{1, K} +1 \right ], \label{eq: MNN MULT weights*} \\
        L \left ( \Phi^{\op{STR}_{m, p, \varrho}^n}_{\varepsilon, K} \right ) & \leq C \left [\log_2 \frac{1}{\varepsilon} + \log_2 \gamma + \log_2 \max \set{1, K} +1\right ]. \label{eq: MNN MULT layers*}
    \end{align}
    Compared to the result presented in \cite[Proposition 3.7]{kutyniok2022theoretical}, we reduce the dimension dependency of order 3 to $\log_2 7$, as expected by applying Strassen's algorithm.
\end{remark}

For square matrices, to avoid using the transpose in the input, we only require to do a minor modification on the definition of the previous $\varrho$-MNN:
\begin{equation}\label{eq:defSTRsquare}
\Phi^{\op{STR}_{n, \varrho}}_{\varepsilon, K} = \Phi_{n, n}^{\op{SHR}_n} \odot \Phi_{\varepsilon, K}^{\op{STR}_{2^{k}, \varrho}} \odot \Phi_n^{\op{EXT}^*},
\end{equation}
where $k = \Ceil{\log_2 n}$ and $\Phi_n^{\op{EXT}^*}$ is a  MNN whose input is in $\R^{n\times 2n}$ and output in $\R^{2^k\times 2\cdot2^k}$ satisfying:
\[
R \left ( \Phi_n^{\op{EXT}^*} \right ) (A | B) = \left . \begin{pmatrix} A & 0^{n, n - 2^k} \\ 0^{2^k - n, n} & 0^{2^k - n, 2^k - n} \end{pmatrix} \middle | \begin{pmatrix} B & 0^{n, n - 2^k} \\ 0^{2^k - n, n} & 0^{2^k - n, 2^k - n} \end{pmatrix} \right .
\]
for all $A, B \in \R^{n \times n}$. 
\begin{remark} 
    When $n=2^k$, $R\left(\Phi^{\op{SHR}_n}_{n,n}\right)$ and $R\left(\Phi_n^{\op{EXT}^*}\right)$ are identity tensors, which is why we consider \eqref{eq:defSTRsquare} an extension. 
    In that case, there is only a small difference in its length and number of weights, which is irrelevant for getting the asymptotic values on the number of weights and layers. 
\end{remark}

Considering Corollary \ref{cor: MNN STR rect properties}, we have that:
\begin{corollary} \label{cl: MNN Strassen for square matrices}
    Let $A, B \in \R^{n \times n}$, and $K > 0$ satisfying:
    \[
    \Abs{A}_\infty, \Abs{B}_\infty \leq K.
    \]
    Then, for every $\varepsilon > 0$,
    \begin{align}
        & \Abs{R\left(\Phi_{\varepsilon, K}^{\op{STR}_{n, \varrho}}\right)\left ( A \middle | B \right ) - AB}_\infty \leq \varepsilon, \label{eq: MNN STR arbitrary square product} \\
        & M \left ( \Phi_{\varepsilon, K}^{\op{STR}_{n, \varrho}} \right ) \leq 7 n^{\log_2 7} \left ( M \left ( \Pi^{\varrho}_{\frac{\varepsilon}{4 n^2}, 2 n K} \right ) + 12 \right ) - 9 n^2, \label{eq: MNN STR arbitrary square weights} \\
        & L \left ( \Phi_{\varepsilon, K}^{\op{STR}_{n, \varrho}} \right ) \leq L \left ( \Pi_{\frac{\varepsilon}{4 n^2}, 2 n K}^\varrho \right ) + 2(\log_2 n + 2). \label{eq: MNN STR arbitrary square layers}
    \end{align}
\end{corollary}
Indeed, the resulting Matrix Neural Network is as in Corollary \ref{cor: MNN STR rect properties} but without transposing.

\section{Matrix Inversion Approximation} \label{sec: Strassen Matrix Multiplication for Inverting Matrices}

In this section, we aim to construct a Matrix Neural Network which, given a matrix input, returns an approximation to its inverse. This can be achieved using our newly defined Strassen Neural Networks to approximate the matrix product and the Neumann series for matrices.

We know that for all $A \in \R^{n \times n}$ such that $\Norm{A}_2 < 1$, we have the identity:
\[
(\op{I}_n - A)^{-1} = \sum_{k = 0}^\infty A^k,
\]
where the right-hand side is called the Neumann series of $A$.
We also have the estimate:
\begin{equation} \label{eq: Neumann partial sum error}
\Norm{(\op{I}_n - A)^{-1} - \sum_{k = 0}^N A^k}_2 = \Norm{\sum_{k = N + 1}^\infty A^k}_2 \leq \sum_{k = N + 1}^{\infty} \Norm{A}^k_2 = \frac{\Norm{A}_2^{N + 1}}{1- \Norm{A}_2}.
\end{equation}
Consequently, we can approximate the Neumann series of a matrix computing the partial sums while minimizing the number of matrix products using the following identity:
\begin{equation}
    \sum_{k = 0}^{2^{N + 1} - 1} A^k  = \left ( A^{2^N}+\op{I}_n  \right ) \sum_{k = 0}^{2^N - 1} A^k  = \prod_{k = 0}^{N} \left ( A^{2^k} + \op{I}_n \right ). \label{eq: Neumann sum as product}
\end{equation}
Indeed, the left hand side of \eqref{eq: Neumann sum as product} has a considerably larger number of multiplications than the logarithmic order on the right-hand side.

\begin{remark} \label{rm: inverse as Neumann sum}
    The Neumann series can be used to approximate the inverse of any invertible matrix $A\in\mathbb{R}^{n\times n}$. If $P\in\mathbb{R}^{n\times n}$ satisfies $\Norm{\op{I}_n - A P}_2 < 1$ then
    \[
    A^{-1} = P(\op{I}_n - (\op{I}_n - AP))^{-1} = P \sum_{k = 0}^\infty (\op{I}_n - AP)^k.
    \]
    Conversely, if such $P$ does not exist, then $A$ cannot be invertible.
    The particular case where $P = \alpha \op{I}$ is of high interest for real symmetric positive definite matrices, as $\alpha$ can be estimated fairly easily. Indeed, if $A$ is real symmetric, it can be diagonalized and 
    $$\Norm{\op{I}_n - \alpha A}_2=\max_{\lambda\in\Lambda(A)}|1-\alpha\lambda|,$$
    for $\Lambda(A)$ the set of eigenvalues of $A$. Real symmetric positive definite matrices appear naturally in the resolution of PDEs with the Galerkin method or finite differences. 
\end{remark}

We can finally get a MNN approximating the inverse of some matrices.

\begin{theorem} \label{th: INV MNN}
    Suppose that Hypothesis \ref{hyp: product varrho realization approximated} is satisfied. Let $\varepsilon > 0$, $\alpha \in \R$, $\delta \in (0, 1)$ and $n \in \N$. We call
    \begin{equation}\label{def:nepsdel}
    \begin{aligned}
        N(\varepsilon, \delta) &= \max \set{\Ceil{\log_2 \left ( \frac{\log_2(\varepsilon(1-\delta))}{\log_2(\delta)} \right )}, 1}, \\
        \Sigma(\varepsilon, \delta, n) &= 2^{-2^{N(\varepsilon, \delta)}} \frac{\min \set{\varepsilon, 1/4}}{16n^3}.
    \end{aligned}
    \end{equation}
    Then, there exists a $\varrho$-MNN $\Phi^{\op{INV}_{n, \varrho}^\alpha}_{\varepsilon, \delta}$ satisfying:
    \begin{equation} \label{eq: INV MNN inverse}
        \Norm{A^{-1} - R \left ( \Phi^{\op{INV}_{n, \varrho}^\alpha}_{\varepsilon, \delta} \right ) (A)}_2 \leq \varepsilon
    \end{equation}
    for all $A \in \R^{n \times n}$ such that 
    \[
    \Norm{\Eye_n - \alpha A}_2 \leq \delta,
    \]
    and satisfying:
    \begin{align}
        M \left ( \Phi^{\op{INV}_{n, \varrho}^\alpha}_{\varepsilon, \delta} \right ) & \leq 14 n^{\log_2 7} \left (N\left ( \frac{\varepsilon}{2 \alpha}, \delta \right ) - 1 \right ) \left ( M \left ( \Pi^\varrho_{\Sigma(\varepsilon/\alpha, \delta, n), 2n} \right ) + 12 \right ) \nonumber \\ 
        & + n^2 \left ( L \left ( \Pi_{\Sigma(\varepsilon/\alpha, \delta, n), 2 n}^\varrho \right ) - 14 N\left ( \frac{\varepsilon}{2 \alpha}, \delta \right ) + 20 + 2 \log_2 n \right ) \label{eq: INV MNN weights} \\
        & + n \left ( N \left ( \frac{\varepsilon}{2 \alpha}, \delta \right ) + 1 \right ), \nonumber \\
        L \left ( \Phi_{\varepsilon, \delta}^{\op{INV}^\alpha_{n, \varrho}} \right ) & \leq N\left ( \frac{\varepsilon}{2 \alpha}, \delta \right ) \left [ 2 \log_2 n + 5 + L \left ( \Pi_{\Sigma(\varepsilon/\alpha, \delta, n), 2 n}^\varrho \right ) \right ], \label{eq: INV MNN layers}
    \end{align}
    if $N \left ( \frac{\varepsilon}{2 \alpha}, \delta \right ) \geq 2$, and satisfying:
    \begin{align}
        M \left ( \Phi^{\op{INV}_{n, \varrho}^\alpha}_{\varepsilon, \delta} \right ) & = 2(n^2 + n), \label{eq: INV MNN weights N equals 1} \\
        L \left ( \Phi^{\op{INV}_{n, \varrho}^\alpha}_{\varepsilon, \delta} \right ) & = 2, \label{eq: INV MNN layers N equals 1}
    \end{align}
    if $N \left ( \frac{\varepsilon}{2 \alpha}, \delta \right ) = 1$.
\end{theorem}

Note that $\varepsilon \mapsto N(\varepsilon,\delta)$ is a decreasing function when $\delta\in(0,1)$.

The proof can be found in Appendix \ref{sec: proof of teo INV MNN}. The idea of the proof is first construct a MNN that approximates the powers $A^{2^k}$, then another to compute the product \eqref{eq: Neumann sum as product}, and finally, using the Neumann series for an appropriate number of terms ($2^{N(\frac{\varepsilon}{2\alpha}, \delta)+1}$ concretely), we get the desired MNN.

\begin{remark}
    As stated in Remark \ref{rm: inverse as Neumann sum}, we could construct a $\varrho$-MNN that approximates any matrix $A \in \R^{n\times n}$ such that
    \[
    \Norm{\op{I}_n - AP} \leq \delta < 1,
    \]
    where is $P \in \R^{n\times n}$ is a fixed invertible matrix. But for simplicity of the statement and proof, we stick to the case $P = \alpha \op{I}_n$.
\end{remark}

\begin{remark} \label{remark: INV MNN comparation}
    For the particular case $\varrho = \ReLU$ and $\alpha = 1$, estimates \eqref{eq: INV MNN weights} and \eqref{eq: INV MNN layers} can be written as:
    \begin{align}
        M \left ( \Phi_{\varepsilon, \delta}^{\op{INV}_n^*} \right ) & \leq C n^{\log_2 7} \log_2 m(\varepsilon, \delta) \left ( \log_2 \frac{1}{\varepsilon} + m(\varepsilon, \delta) + \log_2 n  \right ), \label{eq: MNN INV ReLU weights*} \\
        L \left ( \Phi_{\varepsilon, \delta}^{\op{INV}_n^*} \right ) & \leq C \log_2 m(\varepsilon, \delta) \left ( \log_2 \frac{1}{\varepsilon} + m(\varepsilon, \delta) + \log_2 n \right ), \label{eq: MNN INV ReLU layers*}
    \end{align}
    where $C > 0$ is a constant and
    \[
    m(\varepsilon, \delta) = \frac{\log_2(\varepsilon(1-\delta)/2)}{\log_2 \delta} \approx 2^{N(\varepsilon/2, \delta)}.
    \]
    This result presents various improvements with respect to prior work, concretely \cite[Theorem 3.8]{kutyniok2022theoretical}. Although the number of layers is asymptotically worse, the upper bound on the number of weights on \eqref{eq: MNN INV ReLU weights*} is an improvement over theirs. The exponent on the dimension of the input matrices is reduced from 3 to $\log_2 7$ and the factor is $\log_2 m(\varepsilon, \delta)$ instead of $m(\varepsilon, \delta) \log_2^2(m(\varepsilon, \delta))$. In parentheses, even if we have $m(\varepsilon, \delta)$ instead of $\log_2 m(\varepsilon, \delta)$, it is still an improvement if we compare the expanded expression 
    $m(\varepsilon, \delta)\log_2(m(\varepsilon,\delta))$ to $m(\varepsilon, \delta)\log_2(m(\varepsilon,\delta))\log_2^2(m(\varepsilon,\delta))$. The first change comes from the use of the Strassen algorithm for matrix multiplication, while the second is due to the use of the identity \eqref{eq: Neumann sum as product}.
\end{remark}

\section{\texorpdfstring{Applications for resolution of elliptic \\ parametric PDEs}{Applications for resolution of elliptic parametric PDEs}} \label{sec: Applications}


As explained in the paper \cite{kutyniok2022theoretical}, this way of inverting matrices allows to numerically solve some parametric elliptic PDEs.

Consider the weak form of an elliptic PDE:
\begin{equation} \label{eq: elliptic PDE weak}
b_y(u_y, v) = f_y(v) \quad \txt{ for all } v \in H, y \in \cali{Y},
\end{equation}
where 
\begin{enumerate}
    \item $H$ is a infinite dimensional separable Hilbert space with $\Normm_H$ its associated norm,
    \item $\cali{Y}$ is the parameter set,
    \item $b_y$ is a parameter dependent bilinear form,
    \item $f_y$ is a parameter dependent linear map in the dual of $H$,
    \item $u_y \in H$ is the solution of \eqref{eq: elliptic PDE weak}.
\end{enumerate}
A way to garantee the existence and uniqueness of $u_y$ is assuming some properties on the bilinear forms $b_y$:
\begin{enumerate}
    \item $b_y$ are symmetric, that is, for all $y \in \cali{Y}$ and $u, v \in H$, $b_y(u, v) = b_y(v, u)$.

    \item $b_y$ are uniformly continuous, that is, there exists a constant $C_{\op{cont}} > 0$ such that
    \[
    \Abs{b_y(u, v)} \leq C_{\op{cont}} \Norm{u}_H \Norm{v}_H, \quad \forall y \in \cali{Y},\ \  \forall u, v \in H.
    \]

    \item $b_y$ are uniformly coercive, that is, there exists a constant $C_{\op{coer}} > 0$ such that
    \[
    b_y(u, u) \geq C_{\op{coer}} \Norm{u}_H^2, \quad  \forall y \in \cali{Y},\ \  \forall u \in H.
    \]
\end{enumerate}
In fact, the Lax-Milgram Theorem ensures it, as proven in Theorem 1 of Section 6.2 of \cite{evans2022partial}.

In this situation, we can use the Galerkin method to discretize the Hilbert space truncating a Hilbert base to transform the PDE into a linear system of equations. The associated matrix of said system is defined as
\[
B_y = \left ( b_y(\varphi_i, \varphi_j) \right )_{\substack{i = 1, \dots, n \\ j = 1, \dots, n}} \in \R^{n \times n},
\]
where $\set{\varphi_i}_{i \in \N}$ is an orthonormal basis of $H$ and $n \in \N$ is the number of elements of said basis after the truncation. We are going to see, with the assumption on $b_y$, that the matrices $B_y$ satisfy the requirement that there exist $\alpha \in \R$ and $\delta\in(0,1)$ independent of $y \in \cali{Y}$ such that $\Norm{\Eye_n - \alpha B_y}_2 < \delta$. This property makes the MNN of Theorem \ref{th: INV MNN} particularly useful: one MNN  gives the Galerkin solution for all $y \in \cali{Y}$ as long as we are able to obtain the matrices $B_y$.

\begin{proposition}
    Let $H$ be a separable Hilbert space with $\set{\varphi_i}_{i \in \N}$ an orthonormal basis, $\cali{Y}$ an arbitrary set, and $\set{b_y}_{y \in \cali{Y}}$ be a collection of bilinear forms in $H$. Suppose that $\set{b_y}_{y \in \cali{Y}}$ are symmetric, uniformly continuous with constant $C_{\op{cont}}$ and uniformly coercive with constant $C_{\op{coer}}$. Then, for every $\alpha \in \left (0, \frac{2}{C_{\op{cont}}} \right )$ and $n \in \N$ all matrices
    \[
    B_y = \left ( b_y(\varphi_i, \varphi_j) \right )_{\substack{i = 1, \dots, n \\ j = 1, \dots, n}} \in \R^{n \times n}
    \]
    satisfy: 
    \[
    \Norm{\Eye_n - B_y}_2 \leq \max \set{\Abs{1 - \alpha C_{\op{cont}}}, \Abs{1 - \alpha C_{\op{coer}}}} < 1.
    \]
\end{proposition}

The proof is straight forward, so we chose to omit it.

\begin{remark}
    This proposition is a generalization of Proposition B.1 of \cite{kutyniok2022theoretical}, where they consider $\alpha\in\left(0,\frac{1}{C_{\op{cont}}}\right)$ instead of $\alpha\in\left(0,\frac{2}{C_{\op{cont}}}\right)$.
\end{remark}

We can approximate the inverse of any $B_y$ using the MNN of Theorem \ref{th: INV MNN} for $\alpha \in \left ( 0, \frac{2}{C_{\op{cont}}} \right )$ and $\delta = \max \set{\Abs{1 - \alpha C_{\op{cont}}}, \Abs{1 - \alpha C_{\op{coer}}}}$. Therefore, as explained in \cite{kutyniok2022theoretical}, we can numerically solve \eqref{eq: elliptic PDE weak} for every parameter $y \in \cali{Y}$ with a single MNN. Indeed, they prove that the inverse of all of those matrices can be approximated by a single Neural Network with worse upper bounds on the number of parameters as explained in Remark \ref{remark: INV MNN comparation}. Therefore, we reach the same conclusion with a better dependency on the Galerkin dimension for approximating the parameter-to-solution map. However, other methods, like those of multilevel ideas in \cite{JMLR:v24:23-0421}, provides even better bounds for the particular equation considered where they rely on a more efficient way of solving the resulting linear system.

\section{Conclusions}\label{sec:conclusions}
We have constructed two  Neural Networks approximating matrix multiplication and inversion operators. We have been able to do so without restricting ourselves to a fixed activation function. Also, we bound their number of weights and layers, and by using the Strassen algorithm we have reduced a potential cubic dependency on the dimension to a $\log_2 7$ exponent. Rewriting the Neumann partial sum, we reduce the upper bounds on the number of parameters by reducing the number of matrix multiplication apparitions.
We also explain how the Neural Network that approximates matrix inversion operator can be useful for solving parametric elliptic PDEs. 

As explained, there is a high interest in finding optimal upper bounds on the number of parameters for approximating maps with Neural Networks. A remaining challenge is then finding explicitly the smallest Neural Neural tasked for matrix multiplication or inversion approximation while giving its explicit construction. The answer is unknown even for some simple activation function such as $\ReLU$.


\appendix

\section{Proof of Proposition \ref{prop: MNN STR powers of 2}} \label{sec: proof of MNN STR powers of 2}

\begin{proof}
    The statement \eqref{align: MNN STR product} can be proved inductively on $k$. The base case $k = 0$, when $A$ and $B$ are $1 \times 1$ matrices, that is, scalars, is a direct consequence of \eqref{eq:Phibase} and Hypothesis \ref{hyp: product varrho realization approximated}.

    Let us suppose that Proposition \ref{prop: MNN STR powers of 2} is true for $k - 1$. Let us start by showing \eqref{align: MNN STR product}. Using the Strassen algorithm, we know that \eqref{align: product of matrices as seven products} is satisfied. Consequently, because of the triangular inequality, defining $P_i$ as in \eqref{align: Strassen products}, it suffices to show that:
    \begin{equation} \label{eq: error for the partial product}
        \left | \left [ R \lr{(}{\Phi^{\op{PAR}_{k, \varrho}}_{\varepsilon, K} \odot \Phi^{\op{SPLIT}_k}}{)}(A|B) \right ]_{(i - 1) \cdot 2^{k - 1} + 1:i \cdot 2^{k - 1}, 1: 2^{k - 1}} - P_i \right |_\infty \leq \frac{\varepsilon}{4} \quad \forall i \in \set{1, \dots, 7}.
    \end{equation}
    Let us show the estimate for $i = 1$, as the other six estimates are analogous. Since $\Abs{A}_\infty \leq K$, then $\Abs{A[1, 1] + A[2, 2]}_\infty \leq 2 K$. Similarly, since $\Abs{B}_\infty \leq K$, we have $\Abs{B[1, 1] + B[2, 2]}_\infty \leq 2 K$. Thus, using the recursive hypothesis with $k - 1$ we obtain that:
    \begin{equation*}
        \lr{\lvert}{R \left ( \Phi^{\op{STR}_{2^{k - 1}, \varrho}}_{\frac{\varepsilon}{4}, 2 K} \right ) \left ( A[1, 1] + A[2, 2] \middle | B[1, 1] + B[2, 2] \right )}{.} 
        \lr{.}{ - (A[1, 1] + A[2, 2])(B[1, 1] + B[2, 2])}{\rvert}_\infty \leq \frac{\varepsilon}{4},
    \end{equation*}
    concluding the proof of \eqref{align: MNN STR product}.
    
    To continue with, let us show that the number of weights of the Matrix Neural Network is given by \eqref{align: MNN STR number of weights}.
    We clearly have: \[
    M\left( \Phi_{\varepsilon, K}^{\op{PAR}_{k, \varrho}} \right ) = 7 M \left ( \Phi_{\frac{\varepsilon}{4}, 2K}^{\op{STR}_{2^{k - 1}, \varrho}} \right ).
    \]
    Using the previous equality, \eqref{eq: number of weights mix 1} and \eqref{eq: number of weights SPLIT 1}, we obtain recursively:
    \begin{equation} \label{eq: number of weights of STR MNN with sum}
        \begin{aligned}
            M \left ( \Phi^{\op{STR}_{2^k, \varrho}}_{\varepsilon, K} \right ) & = M \left ( \Phi_{\varepsilon, K}^{\op{PAR}_{k, \varrho}} \right ) + M \left ( \Phi^{\op{MIX}_k} \right ) + M \left ( \Phi^{\op{SPLIT}_k} \right ) \\
            & = 7 M \left ( \Phi_{\frac{\varepsilon}{4}, 2K}^{\op{STR}_{2^{k - 1}, \varrho}} \right ) + 9 \cdot 4^k \\
            & = 7 \left ( 7 M \left ( \Phi_{\frac{\varepsilon}{4^2}, 2^2K}^{\op{STR}_{2^{k - 2}, \varrho}} \right ) + 9 \cdot 4^{k - 1} \right ) + 9 \cdot 4^k \\
            & = \cdots \\
            & = 7^k M \left ( \Pi^{\varrho}_{\frac{\varepsilon}{4^k}, 2^k K} \right ) + 9 \sum_{i = 1}^k 4^i \cdot 7^{k - i}.
        \end{aligned}
    \end{equation}
    We now calculate the sum using the fact that it is geometric:
    \begin{equation} \label{eq: computation sum}
    9 \sum_{i = 1}^k 4^i \cdot 7^{k - i}=
        9 \cdot 7^k \sum_{i = 1}^k \left ( \frac{4}{7} \right )^i = 9 \cdot 7^k \frac{4 / 7 - (4 / 7)^{k + 1}}{1 - 4 / 7} = 12 \left (  7^k - 4^k \right ).
    \end{equation}
    Putting together \eqref{eq: number of weights of STR MNN with sum} and \eqref{eq: computation sum}, we obtain \eqref{align: MNN STR number of weights}.

    To finish, let us obtain the length of the Matrix Neural Network \eqref{align: MNN STR number of layers}. For that purpose, we remark that for $k \geq 1$:
    \[
    L\left ( \Phi^{\op{STR}_{2^{k}, \varrho}}_{\varepsilon, K} \right ) = L \left ( \Phi_{\frac{\varepsilon}{4}, 2 K}^{\op{STR}_{2^{k - 1}, \varrho}} \right ) + 2.
    \]
    Thus, using induction, we conclude with \eqref{align: MNN STR number of layers}.
\end{proof}

\section{Proof of Theorem \ref{th: INV MNN}} \label{sec: proof of teo INV MNN}

The first step is to define a MNN that approximates the function of squaring a matrix $N$ times:
\begin{lemma} \label{lm: SQR N times MNN}
    Suppose that Hypothesis \ref{hyp: product varrho realization approximated} is satisfied, and let $\varepsilon \in (0, \frac{1}{4})$ and $n, N \in \N$. We define $\Phi^{\op{DUP}_n}$ as the shallow MNN satisfying:
    \[
    R \left ( \Phi^{\op{DUP}_n} \right ) (A) = A | A,
    \]
    and  $\Phi_{\varepsilon}^{\op{SQR}_{n, \varrho}^N}$ as the $\varrho$-MNN satisfying:
    \begin{equation} \label{eq: SQR N times MNN definition}
        \Phi_{\varepsilon}^{\op{SQR}_{n, \varrho}^N} = \overbrace{\left(\Phi^{\op{STR}_{n, \varrho}}_{\frac{\varepsilon}{4 n}, 1} \odot \Phi^{\op{DUP}_n}\right) \odot \dots \odot \left(\Phi^{\op{STR}_{n, \varrho}}_{\frac{\varepsilon}{4 n}, 1} \odot \Phi^{\op{DUP}_n}\right)}^{N}.
    \end{equation}
    Then,
    \begin{equation} \label{eq: SQR N times MNN N times square}
    \Norm{A^{2^N} - R \left ( \Phi_{\varepsilon}^{\op{SQR}_{n, \varrho}^N} \right ) (A)}_2 \leq \varepsilon
    \end{equation}
    for all $A \in \R^{n \times n}$ such that $\Norm{A}_2 \leq \frac{1}{2}$.
\end{lemma}

\begin{remark}
    If we define instead the above MNN as
    \[
    \Phi_{\varepsilon}^{\op{SQR}_{n, \varrho}^N} = \overbrace{\left(\Phi^{\op{STR}_{n, \varrho}}_{\frac{\varepsilon}{4 n}, \varepsilon+\frac{1}{2^{2^{N-1}}}} \odot \Phi^{\op{DUP}_n}\right) \odot \dots \odot \left(\Phi^{\op{STR}_{n, \varrho}}_{\frac{\varepsilon}{4 n},\varepsilon+\frac{1}{2^{2}}} \odot \Phi^{\op{DUP}_n}\right)}^{N - 1} \odot \Phi_{\frac{\varepsilon}{n}, \frac{1}{2}}^{\op{STR}_{n, \varrho}} \odot \Phi^{\op{DUP}_n},
    \]
    it satisfies the same property. This can be seen by analyzing in detail the computations of the proof of Lemma \ref{lm: SQR N times MNN}.
    However, for clarity and coherence with the result needed for Lemma \ref{th: NEU MNN} below, we keep the subindices homogeneous and the second one equal to $1$. 
\end{remark}

\begin{proof}
    We can prove \eqref{eq: SQR N times MNN N times square} by induction on $N$. We fix $A \in \R^{n \times n}$ with $\Norm{A}_2 \leq \frac{1}{2}$.
    
    The case $N = 1$ is obvious using \eqref{eq: MNN STR arbitrary square product} and 
\eqref{eq: infinity and 2 norm equivalence constants}.
In fact, we have: \begin{equation} \label{eq: SQR N times MNN N equals 1}
        \begin{aligned}
            \Norm{R \left ( \Phi_{\varepsilon}^{\op{SQR}_{n, \varrho}^1} \right ) (A) - A^2}_2 & = \Norm{R \left ( \Phi_{\frac{\varepsilon}{4 n}, 1}^{\op{STR}_{n, \varrho}} \right ) (A | A) - A^2}_2 \\
            & \leq n \Abs{R \left ( \Phi_{\frac{\varepsilon}{4 n},1}^{\op{STR}_{n, \varrho}} \right ) (A | A) - A^2}_\infty \leq n \frac{\varepsilon}{4 n} = \frac{\varepsilon}{4} < \varepsilon.
        \end{aligned}
    \end{equation}
    
    Suppose now that the statement \eqref{eq: SQR N times MNN N times square} holds for $N \in \N$ and let us prove the case $N + 1$. By \eqref{eq: SQR N times MNN definition}, we know that:
    \[
    \Phi_{\varepsilon}^{\op{SQR}^{N + 1}_{n, \varrho}} = \Phi_\varepsilon^{\op{SQR}^1_{n, \varrho}} \odot \Phi_{\varepsilon}^{\op{SQR}^N_{n, \varrho}}.
    \]
    By the induction hypothesis, we obtain that:
    \begin{equation} \label{eq: inductive upper bound norm}
    \Norm{R \left ( \Phi_{\varepsilon}^{\op{SQR}^N_{n, \varrho}} \right )(A)}_2 \leq \varepsilon + \Norm{A}_2^{2^N} \leq \varepsilon + \frac{1}{2^{2^N}} \leq \frac{1}{2}.
    \end{equation}
    Consequently,
    \begin{align*}
        & \Norm{A^{2^{N + 1}} - R \left ( \Phi_{\varepsilon}^{\op{STR}^{N + 1}_{n, \varrho}} \right )(A)}_2 \\
        & \leq \Norm{\left ( A^{2^N} \right )^2 - \left ( R \left ( \Phi_{\varepsilon}^{\op{SQR}^{N }_{n, \varrho}} \right ) (A) \right )^2}_2 \\
        & + \Norm{\left ( R \left ( \Phi_{\varepsilon}^{\op{SQR}^{N}_{n, \varrho}} \right ) (A) \right )^2 - R \left ( \Phi_{\varepsilon}^{\op{SQR}^{N + 1}_{n, \varrho}} \right ) (A)}_2 \\
        & = \Norm{A^{2^N} \left [ A^{2^N} - R \left ( \Phi_{\varepsilon}^{\op{SQR}^{N }_{n, \varrho}} \right ) (A) \right ] + \left [ A^{2^N} - R \left ( \Phi_{\varepsilon}^{\op{SQR}^{N }_{n, \varrho}} \right ) (A) \right ] R \left ( \Phi_{\varepsilon}^{\op{SQR}^{N }_{n, \varrho}} \right ) (A)}_2 \\
        & + \Norm{\left ( R \left ( \Phi_{\varepsilon}^{\op{SQR}^{N}_{n, \varrho}} \right ) (A) \right )^2 - R \left (\Phi_{\varepsilon}^{\op{SQR}^1_{n, \varrho}} \right ) \left ( R \left ( \Phi_{\varepsilon}^{\op{SQR}^{N}_{n, \varrho}} \right ) (A) \right )}_2 \\
        & \leq \Norm{A^{2^N} - R \left ( \Phi_{\varepsilon}^{\op{SQR}^{N}_{n, \varrho}} \right ) (A)}_2 \left (\Norm{A^{2^N}}_2 + \Norm{R \left ( \Phi_{\varepsilon}^{\op{SQR}^{N }_{n, \varrho}} \right ) (A)}_2 \right ) + \frac{\varepsilon}{4} \\
        & \leq \varepsilon \left ( \Norm{A}_2^{2^N} + \Norm{R \left ( \Phi_{\varepsilon}^{\op{SQR}^{N}_{n, \varrho}} \right ) (A)}_2 \right ) + \frac{\varepsilon}{4} \\
        & \leq \varepsilon \left ( \frac{1}{2^{2^N}} + \frac{1}{2} \right ) + \frac{\varepsilon}{4} \leq \varepsilon,
    \end{align*}
    where we used the inductive hypothesis \eqref{eq: SQR N times MNN N times square} and also used \eqref{eq: SQR N times MNN N equals 1} since \eqref{eq: inductive upper bound norm} holds.
\end{proof}


Now, we construct a MNN that approximates the product \eqref{eq: Neumann sum as product}.

\begin{lemma} \label{th: NEU MNN}
    Suppose that Hypothesis \ref{hyp: product varrho realization approximated} is satisfied and let $\varepsilon \in (0, \frac{1}{8})$ and $n, N \in \N$. Then, if $N \geq 2$ there exists a $\varrho$-MNN $\Phi_{\varepsilon}^{\op{NEU}_{n, \varrho}^N}$ satisfying:
    \begin{equation} \label{eq: NEU MNN Neumann series}
    \Norm{\sum_{k = 0}^{2^N - 1} A^k - R \left ( \Phi_{\varepsilon}^{\op{NEU}_{n, \varrho}^N} \right ) (A)}_2 \leq \varepsilon
    \end{equation}
    for all $A \in \R^{n \times n}$ such that $\Norm{A}_2 \leq 1$, satisfying:
    \begin{equation}
    \begin{split}
        M \left ( \Phi_{\varepsilon}^{\op{NEU}_{n, \varrho}^N} \right ) & \leq 14 n^{\log_2 7} (N - 1) \left ( M \left ( \Pi_{2^{-2^N} \frac{\varepsilon}{8 n^3}, 2n}^\varrho \right ) + 12 \right ) \\
        & + n^2 \left ( L \left ( \Pi_{2^{-2^N} \frac{\varepsilon}{8 n^3}, 2n}^\varrho \right ) - 14 N + 19 + 2 \log_2 n\right ) + n N, 
    \end{split} \label{eq: NEU MNN weights}
    \end{equation}
    and
    \begin{equation}
        L \left ( \Phi_{\varepsilon}^{\op{NEU}_{n, \varrho}^N} \right )  \leq N \left [ 2\log_2 n + 5 + L \left ( \Pi_{2^{-2^N} \frac{\varepsilon}{8 n^3}, 2 n}^\varrho \right ) \right]. \label{eq: NEU MNN layers}
    \end{equation}
    Also, if $N = 1$  there exists a $\varrho$-MNN $\Phi_{\varepsilon}^{\op{NEU}_{n, \varrho}^N}$ satisfying \eqref{eq: NEU MNN Neumann series} for $\varepsilon=0$ and all $A\in\R^{n\times n}$, and satisfying:
    \begin{align}
        M \left ( \Phi_{\varepsilon}^{\op{NEU}_{n, \varrho}^1} \right ) & = n^2 + n,\label{eq: NEU MNN N equals 1 weights} \\
        L \left ( \Phi_{\varepsilon, K}^{\op{NEU}_{n, \varrho}^1} \right ) & = 1. \label{eq: NEU MNN N equals 1 layers}
    \end{align}
\end{lemma}

Part of the improvements with respect to the paper \cite{kutyniok2022theoretical} can be found in the Master Thesis \cite{romera2024neuralnetworksnumericalanalysis}

\begin{proof}
    We split the proof into 6 steps:
    \begin{itemize}
        \item Step 1: Proof of $N = 1$.
        \item Step 2: Informal definition of the  MNN.
        \item Step 3: Formal definition of the  MNN.
        \item Step 4: Proof of \eqref{eq: NEU MNN Neumann series} for $N \geq 2$.
        \item Step 5: Proof of \eqref{eq: NEU MNN weights} for $N \geq 2$.
        \item Step 6: Proof of \eqref{eq: NEU MNN layers} for $N \geq 2$.
    \end{itemize}
    
    \vspace{.5cm}
    \textbf{Step 1:} Proof of $N = 1$. We define $\Phi^{\op{NEU}^1_{n, \varrho}}_{\varepsilon} = ((\op{id}, \Eye_n, \op{id}))$, from where the result follows easily.    

    \textbf{Step 2:} Informal definition of the MNN.
    We now suppose that $N \geq 2$. The general idea is to use the product formula introduced in \eqref{eq: Neumann sum as product} for the Neumann series.
    However, by our construction of $\Phi^{\op{STR}}$, we require to bound the norm of the approximation of $\Eye_n + A^{2^k}$ and $A^{2^k}$ for $k \in \set{0, \dots, N}$ for an accurate product. For that purpose, we rewrite \eqref{eq: Neumann sum as product} as follows:
    \begin{equation} \label{eq: Neumann series as product with factor}
    \prod_{k = 0}^N  \left ( A^{2^k} + \Eye_n  \right ) = 2^{\sum_{k = 0}^N 2^k}\prod_{k = 0}^N 2^{-2^k} \left (A^{2^k}+ \Eye_n \right ) = 2^{2^{N + 1} - 1} \prod_{k = 0}^N \left[\left(\frac{A}{2}\right)^{2^k}+\left(\frac{\Eye_n}{2}\right)^{2^k} \right].
    \end{equation}
    We do this to have upper bounds of the approximations with the cost of the power of 2 as a factor.
    \paragraph{}
    The target MNN will then have two branches: one branch responsible for approximating $\left (\frac{A}{2} \right )^{2^k}$ for each $k \in \set{0, \dots, N}$ like in Lemma \ref{lm: SQR N times MNN}, while the other approximates the target product using recursively the different exponents of $\frac{A}{2}$.
    Our plan is then to follow the sketch:
    \begin{equation} \label{eq: informal sketch}
        \begin{aligned}
            A & \mapsto \begin{pmatrix} \frac{A}{2} & | & \frac{A}{2} \\ \frac{A}{2} & | & 0 \end{pmatrix} \mapsto \begin{pmatrix} \left (\frac{A}{2} \right )^2 \\ \frac{A}{2} + \frac{\Eye_n}{2} \end{pmatrix} 
            \mapsto \begin{pmatrix} \left ( \frac{A}{2} \right )^4 \\  \left ( \frac{A}{2} + \frac{\Eye_n}{2} \right)\left ( \left ( \frac{A}{2} \right )^2 + \frac{1}{2^2}\Eye_n  \right) \end{pmatrix} \mapsto \cdots 
            \\&\mapsto \begin{pmatrix}\left(\frac{A}{2}\right)^{2^N}\\ \prod_{k = 0}^{N - 1} \left ( \left ( \frac{A}{2} \right )^{2^k} + \frac{1}{2^{2^k}}\Eye_n \right )
             \end{pmatrix}   
            \mapsto 
             2^{2^{N + 1} - 1} \prod_{k = 0}^{N} \left ( \left ( \frac{A}{2} \right )^{2^k} + \frac{1}{2^{2^k}}\Eye_n \right). 
        \end{aligned}
    \end{equation}
    In this way, we compute $\left(\frac{A}{2}\right)^{2^k}$ only once for every $k \in \N$.

    \textbf{Step 3:} Formal definition of the MNN.    
    We define some  auxiliary MNN satisfying:
    \begin{align*}
        R \left ( \Phi^{\op{DUP}}_n \right )(A) & = \frac{1}{2} \begin{pmatrix} A & | & A \\ A & | & 0 \end{pmatrix}, \\
        R \left ( \Phi^{\op{FILL}}_{n, L} \right )(A | B) & = A + \frac{\Eye_n}{2}, \\
        R \left ( \Phi^{\op{FLIP}}_{n, k} \right ) \begin{pmatrix} A \\ B \end{pmatrix} & = (A + 2^{-2^k} \Eye_n )| B, \\
        R \left ( \Phi^{\op{MIX}}_{n, k} \right ) \begin{pmatrix} A \\ B \end{pmatrix} & = \begin{pmatrix} A & | & A \\ A + 2^{-2^k} \Eye_n & | & B \end{pmatrix},
    \end{align*}
    for all $A, B \in \R^{n \times n}$, $k, L \in \N$ with
    \begin{equation} \label{eq: number of weights and layers of auxiliary MNN upper bounds}
    \begin{aligned}
        M \left ( \Phi^{\op{DUP}}_n \right ) & = 3 n^2, & L \left ( \Phi^{\op{DUP}}_n \right ) & = 1, \\
        M \left ( \Phi^{\op{FILL}}_{n, L} \right ) & = n^2 L + n, & L \left ( \Phi^{\op{FILL}}_{n, L} \right ) & = L, \\
        M \left ( \Phi^{\op{FLIP}}_{n, k} \right ) & = 2 n^2 + n, & L \left ( \Phi^{\op{FLIP}}_{n, k} \right ) & = 1,\\
        M \left ( \Phi^{\op{MIX}}_{n, k} \right ) & = 4 n^2 + n, & L \left ( \Phi^{\op{MIX}}_{n, k} \right ) & = 1.
    \end{aligned}
    \end{equation}
    The MNN $\Phi_{n, L}^{\op{FILL}}$ is artificially elongated in order to parallelize it with a MNN with more than 1 layer.
    We can then define recursively that:
    \begin{equation}\label{eq:defPhiAux}
    \begin{cases}
        \Phi^{\op{AUX}^1_{n, \varrho}}_{\varepsilon}  = \cali{P} \left ( \Phi^{\op{STR}_{n, \varrho}}_{\frac{\varepsilon}{4 n}, 1}, \Phi^{\op{FILL}}_{n, L\left ( \Phi^{\op{STR}_{n, \varrho}}_{\frac{\varepsilon}{4 n}, 1} \right )} \right ) \odot \Phi^{\op{DUP}}_n, \\
        \Phi^{\op{AUX}^i_{n, \varrho}}_{\varepsilon}  = \cali{P} \left ( \Phi^{\op{STR}_{n, \varrho}}_{\frac{\varepsilon}{4 n}, 1}, \Phi^{\op{STR}_{n, \varrho}}_{\frac{\varepsilon}{4 n}, 1} \right ) \odot \Phi^{\op{MIX}}_{n, i - 1} \odot \Phi^{\op{AUX}^{i - 1}_{n, \varrho}}_{\varepsilon} &\forall i\geq2.
    \end{cases} 
    \end{equation}
    If we define: 
    \[
    \left ( (\LL^1, C^1, \alpha^1), \dots, (\LL^L, C^L, \alpha^L) \right ) \coloneq \Phi^{\op{STR}_{n, \varrho}}_{2^{1-2^{N}} \frac{\varepsilon}{4 n},1} \odot \Phi^{\op{FLIP}}_{n, N - 1} \odot \Phi^{\op{AUX}^{N - 1}_{n, \varrho}}_{2^{1-2^N} \varepsilon},
    \] 
    where we have replaced $\varepsilon$ by $2^{1-2^{N}} \varepsilon$, then the target $\varrho$-MNN is:
    \[
    \Phi_\varepsilon^{\op{NEU}^N_{n, \varrho}} \coloneq \left ( (\LL^1, C^1, \alpha^1), \dots, \left ( 2^{2^N-1} \LL^L, 2^{2^N-1} C^L, \alpha^L \right ) \right ).
    \]

    \paragraph{}
    For legibility, we introduce some notation useful for the next steps:
    \begin{align}
        \op{SQR}_\varepsilon^i(A) & = \left [ R \left ( \Phi_\varepsilon^{\op{AUX}^i_{n, \varrho}} \right ) (A) \right ]_{1: n, 1: n},\label{eq:sqrepsdefalt} \\
        \op{PROD}_\varepsilon^i(A) & = \left [ R \left ( \Phi_\varepsilon^{\op{AUX}^i_{n, \varrho}} \right ) (A) \right ]_{(n + 1): 2n, 1: n},\notag \\
        P^i(A) & = \prod_{k = 0}^i \left( \left(\frac{A}{2}\right)^{2^k}+  \left(\frac{\Eye_n}{2}\right)^{2^k} \right )\notag
    \end{align}
    for all $i\in \N$ and $A \in \R^{n \times n}$. It is not difficult to see using \eqref{eq: SQR N times MNN definition} and \eqref{eq:defPhiAux} that:
    \begin{equation}\label{eq:SqriA=R()}
    \op{SQR}^i_\varepsilon (A) = R \left ( \Phi_{\varepsilon}^{\op{SQR}^i_{n, \varrho}}\right )\left(\frac{A}{2}\right),
    \end{equation}
    where $\Phi_\varepsilon^{\op{SQR}^i_{n, \varrho}}$ is the $\varrho$-MNN from Lemma \ref{lm: SQR N times MNN}.
    Thus, we know by \eqref{eq: SQR N times MNN N times square} that: \begin{equation} \label{eq: SQR norm 2 upper bound}
        \Norm{\op{SQR}^i_\varepsilon(A)}_2 \leq \varepsilon + \Norm{\frac{A}{2}}_2^{2^i} \leq \varepsilon + 2^{-2^i}
    \end{equation}
    for all $i\in \N$ and $A \in \R^{n \times n}$ such that $\Norm{A}_2 \leq 1$. It is also easy to check that:
    \begin{equation} \label{eq: PN norm 2 upper bound}
        \Norm{P^i(A)}_2 \leq \prod_{k = 0}^i \frac{2}{2^{2^k}} = 2^{i + 2 - 2^{i + 1}} \leq \frac{1}{2}
    \end{equation}
    for all $i \in \N$ and $A \in \R^{n \times n}$ with $\Norm{A}_2 \leq 1$. 

    \begin{figure}[t]
        \centering
        \begin{tikzpicture}
            \node (A) at (0, 0) {$A$};
            \node[node distance = 24pt, draw, minimum height = 96pt, minimum width = 48pt] (DUP) [right=of A] {};
            \draw[|-] (A.east) -- (DUP.west);
            \node[above = 16pt of DUP.east] (aboveDUPeast) {};
            \node[below = 16pt of DUP.east] (belowDUPeast) {};
            \node[node distance = 24pt] (midAUX1) [right=of DUP.east] {};
            \path (midAUX1) -- node[font = \tiny] (AA) {$\frac{1}{2} A | A$} (midAUX1.center |- DUP.north);
            \path (midAUX1) -- node[font = \tiny] (A0) {$\frac{1}{2} A | 0$} (midAUX1.center |- DUP.south);
            \draw[->, anchor = west] (DUP.east |- AA) -- (AA);
            \draw[->, anchor = west] (DUP.east |- A0) -- (A0);
            \node[node distance = 24pt, draw, minimum height = 96pt, minimum width = 48pt, anchor = left] (STR1FILL) [right=of midAUX1] {};
            \path (STR1FILL.center) -- node[inner sep = 0pt] (STR1) {$\Phi^{\op{STR}_{n, \varrho}}_{\frac{\varepsilon}{4 n}, \frac{1}{2}}$} (STR1FILL.north);
            \path (STR1FILL.center) -- node[inner sep = 0pt] (FILL) {$\Phi^{\op{FILL}}_{n, L}$} (STR1FILL.south);
            \draw[dashed] (STR1FILL.west) -- (STR1FILL.east);
            \draw[densely dotted, color = gray, rounded corners] (DUP.west) -- (DUP.center) -- (DUP.center |- AA) -- (DUP.east |- AA);
            \draw[densely dotted, color = gray, rounded corners] (DUP.west) -- (DUP.center) -- (DUP.center |- A0) -- (DUP.east |- A0);
            \node[node distance = 24pt] at (DUP) {$\Phi^{\op{DUP}}_n$};
            \draw[|-] (AA) -- (STR1FILL.west |- AA);
            \draw[|-] (A0) -- (STR1FILL.west |- A0);
            \node[node distance = 32pt] (endAUX1) [right=of STR1FILL.east] {};
            \node[font = \tiny] (SQR1) at (endAUX1 |- STR1) {$\op{SQR}^1_\varepsilon(A)$};
            \node[font = \tiny] (PROD1) at (endAUX1 |- FILL) {$\op{PROD}^1_\varepsilon(A)$};
            \draw[->] (STR1FILL.east |- STR1) -- (SQR1);
            \draw[->] (STR1FILL.east |- FILL) -- (PROD1);
            \node[node distance = 64 pt, draw, minimum height = 96pt, minimum width = 48pt] (MIX1) [below=of DUP.center] {};
            \node[node distance = 56 pt] (midAUX2) [right=of MIX1.east] {};
            \path (midAUX2) -- node[font = \tiny] (SQR1SQR1) {$\op{SQR}^1_\varepsilon(A) | \op{SQR}^1_\varepsilon(A)$} (midAUX2 |- MIX1.north);
            \path (midAUX2) -- node[font = \tiny] (SQR1PROD1) {$\op{SQR}^1_\varepsilon(A) + \frac{\Eye_n}{2^{2^1}} | \op{PROD}^1_\varepsilon(A)$} (midAUX2 |- MIX1.south);
            \path (DUP.south) -- node[pos = .33] (AUX121) {} (MIX1.north);
            \path (DUP.south) -- node[pos = .67] (AUX122) {} (MIX1.north);
            \path (PROD1.east |- STR1FILL) -- node[pos = 1] (AUX123) {} +(12pt, 0);
            \path (MIX1.west) -- node[pos = 1] (AUX124) {} +(-12pt, 0);
            \path (PROD1.east |- STR1FILL) -- node[pos = 1] (AUX125) {} +(24pt, 0);
            \path (MIX1.west) -- node[pos = 1] (AUX126) {} +(-24pt, 0);
            \draw[|-, rounded corners] (SQR1) -- (PROD1.east |- SQR1) -- (AUX123 |- SQR1) -- (AUX123 |- AUX121) -- (MIX1.west |- AUX121) -- (AUX124 |- AUX121) -- (AUX124 |- SQR1SQR1) -- (MIX1.west |- SQR1SQR1);
            \draw[|-, rounded corners] (PROD1) -- (AUX125 |- PROD1) -- (AUX125 |- AUX122) -- (AUX126 |- AUX122) -- (AUX126 |- SQR1PROD1) -- (MIX1.west |- SQR1PROD1);
            \draw[->] (MIX1.east |- SQR1SQR1) -- (SQR1SQR1);
            \draw[->] (MIX1.east |- SQR1PROD1) -- (SQR1PROD1);
            \draw[color = gray, densely dotted, rounded corners] (MIX1.west |- SQR1SQR1) -- (MIX1.east |- SQR1SQR1);
            \draw[color = gray, densely dotted, rounded corners] (MIX1.west |- SQR1SQR1) -- (MIX1 |- SQR1SQR1) -- (MIX1 |- SQR1PROD1) -- (MIX1.east |- SQR1PROD1);
            \draw[color = gray, densely dotted] (MIX1.west |- SQR1PROD1) -- (MIX1.east |- SQR1PROD1);
            \node at (MIX1) {$\Phi^{\op{MIX}}_{n, 1}$};
            \node[node distance = 56pt, draw, minimum height = 96pt, minimum width = 48pt, anchor = center] (STRSTR) [right=of midAUX2] {};
            \draw[|-] (SQR1SQR1) -- (STRSTR.west |- SQR1SQR1);
            \draw[|-] (SQR1PROD1) -- (STRSTR.west |- SQR1PROD1);
            \draw[dashed] (STRSTR.west) -- (STRSTR.east);
            \path (STRSTR.center) -- node[inner sep = 0] (STR2) {$\Phi_{\frac{\varepsilon}{4n}, 1}^{\op{STR}_{n, \varrho}}$} (STRSTR.north);
            \path (STRSTR.center) -- node[inner sep = 0] (STR3) {$\Phi_{\frac{\varepsilon}{4n}, 1}^{\op{STR}_{n, \varrho}}$} (STRSTR.south);
            \node[node distance = 32pt] (endAUX2) [right=of STRSTR.east] {};
            \path (endAUX2.center) -- node[font = \tiny] (SQR2) {$\op{SQR}^2_\varepsilon(A)$} (endAUX2 |- STRSTR.north);
            \path (endAUX2.center) -- node[font = \tiny] (PROD2) {$\op{PROD}^2_\varepsilon(A)$} (endAUX2 |- STRSTR.south);
            \draw[->] (STRSTR.east |- SQR2) -- (SQR2);
            \draw[->] (STRSTR.east |- PROD2) -- (PROD2);
            \node[node distance = 64pt, draw, minimum height = 96pt, minimum width = 48pt] (MIXN) [below=of MIX1.center] {};
            \node[node distance = 56pt] (midAUXN) [right=of MIXN.east] {};
            \path (midAUXN) -- node[font = \tiny] (SQRNSQRN) {$\op{SQR}^{i-1}_\varepsilon(A) | \op{SQR}^{i-1}_\varepsilon(A)$} (midAUXN |- MIXN.north);
            \path (midAUXN) -- node[font = \tiny] (SQRNPRODN) {\scalebox{.8}{$\op{SQR}^{i-1}_\varepsilon(A) + \frac{\Eye_n}{2^{2^{i-1}}} \Big| \op{PROD}^{i-1}_\varepsilon(A)$}} (midAUXN |- MIXN.south);
            \path (PROD2.east |- MIXN.north) -- node[pos = .2] (AUXN1) {} node[pos = .4] (AUXN2) {} node[pos = .6] (AUXN3) {} node[pos = .8] (AUXN4) {} (MIXN.north west);
            \path (MIXN.north) -- node[pos = .33] (AUXN5) {} node[pos = .67] (AUXN6) {} (MIX1.south);
            \path (PROD2.east |- STRSTR) -- node[pos = .5] (AUXN7) {} node[pos = 1] (AUXN8) {} +(24pt, 0);
            \draw[|-, rounded corners] (PROD2) -- (AUXN8 |- STR3) -- (AUXN8 |- AUXN5) -- (AUXN1 |- AUXN5);
            \draw[loosely dash dot] (AUXN1 |- AUXN5) -- (AUXN2 |- AUXN5);
            \draw[loosely dash dot] (AUXN3 |- AUXN5) -- (AUXN4 |- AUXN5);
            \path (MIXN.west) -- node[pos = .5] (AUXN9) {} node[pos = 1] (AUXN10) {} +(-24pt, 0);
            \draw[rounded corners] (AUXN4 |- AUXN5) -- (AUXN10 |- AUXN5) -- (AUXN10 |- SQRNPRODN) -- (MIXN.west |- SQRNPRODN);
            \draw[|-, rounded corners] (SQR2) -- (AUXN7 |- STR2) -- (AUXN7 |- AUXN6) -- (AUXN1 |- AUXN6);
            \draw[loosely dash dot] (AUXN1 |- AUXN6) -- (AUXN2 |- AUXN6);
            \draw[loosely dash dot] (AUXN3 |- AUXN6) -- (AUXN4 |- AUXN6);
            \draw[rounded corners] (AUXN4 |- AUXN6) -- (AUXN9 |- AUXN6) -- (AUXN9 |- SQRNSQRN) -- (MIXN.west |- SQRNSQRN);
            \draw[color = gray, densely dotted] (MIXN.west |- SQRNSQRN) -- (MIXN.east |- SQRNSQRN);
            \draw[color = gray, densely dotted] (MIXN.west |- SQRNPRODN) -- (MIXN.east |- SQRNPRODN);
            \draw[color = gray, densely dotted, rounded corners] (MIXN.west |- SQRNSQRN) -- (MIXN.center |- SQRNSQRN) -- (MIXN.center |- SQRNPRODN) -- (MIXN.east |- SQRNPRODN);
            \node at (MIXN) {$\Phi_{n, i-1}^{\op{MIX}}$};
            \draw[->] (MIXN.east |- SQRNSQRN) -- (SQRNSQRN);
            \draw[->] (MIXN.east |- SQRNPRODN) -- (SQRNPRODN);
            \node[node distance = 56pt, draw, minimum height = 96pt, minimum width = 48pt] (STRSTR2) [right=of midAUXN] {};
            \draw[|-] (SQRNSQRN) -- (STRSTR2.west |- SQRNSQRN);
            \draw[|-] (SQRNPRODN) -- (STRSTR2.west |- SQRNPRODN);
            \draw[dashed] (STRSTR2.west) -- (STRSTR2.east);
            \path (STRSTR2.center) -- node[inner sep = 0, anchor = mid] (STR4) {$\Phi_{\frac{\varepsilon}{4n}, 1}^{\op{STR}_{n, \varrho}}$} (STRSTR2.north);
            \path (STRSTR2.center) -- node[inner sep = 0, anchor = mid] (STR5) {$\Phi_{\frac{\varepsilon}{4n}, 1}^{\op{STR}_{n, \varrho}}$} (STRSTR2.south);
            \node[node distance = 64pt] (endAUXN) [right=of STRSTR2.east] {};
            \path (endAUXN.center) -- node (SQRN) {$\op{SQR}^{i}_\varepsilon(A)$} (endAUXN |- STRSTR2.north);
            \path (endAUXN.center) -- node (PRODN) {$\op{PROD}^{i}_\varepsilon(A)$} (endAUXN |- STRSTR2.south);
            \draw[->] (STRSTR2.east |- SQRN) -- (SQRN);
            \draw[->] (STRSTR2.east |- PRODN) -- (PRODN);
            \path (AUXN2 |- AUXN5) -- node {$\cdots$} (AUXN3 |- AUXN6);
        \end{tikzpicture}
        \caption{Matrix Neural Network scheme of $\Phi^{\op{AUX}^{i}_{n, \varrho}}_\varepsilon$ for all $i\geq2$.}
        \label{fig: MNN AUX scheme}
    \end{figure}

    Figure \ref{fig: MNN AUX scheme} visualizes the process that follows an input matrix $A$ in the MNN $\Phi^{\op{AUX}^{i}_{n, \varrho}}_\varepsilon$. Basically, the operations that appear in the informal sketch \eqref{eq: informal sketch} are replaced by the approximations performed by the MNNs introduced previously in a systematic way.

    \textbf{Step 4:} Proof of \eqref{eq: NEU MNN Neumann series} for $N \geq 2$.
  Equations \eqref{eq: SQR N times MNN N times square} and \eqref{eq:SqriA=R()} give us
    \[
    \Norm{\left(\frac{A}{2}\right)^{2^i} - \op{SQR}^i_\varepsilon(A)}_2 \leq \varepsilon
    \]
    for all $i\in\N$ and $A \in \R^{n \times n}$ such that $\Norm{A}_2 \leq 1$.
    
    Let us see by induction over $i \geq 2$ that:
    \begin{equation} \label{eq: MNN AUX approximates Neu except by constant}
        \Norm{P^{i-1}(A) - \left [ R \left ( \Phi^{\op{AUX}^i_{n, \varrho}}_{\varepsilon} \right )(A) \right ]_{n + 1: 2n, 1: n}}_2 \leq \varepsilon
    \end{equation}
    for all $A \in \R^{n \times n}$ such that $\Norm{A}_2 \leq 1$. 

    We start with the base case $i = 2$. Let $A \in \R^{n \times n}$ be such that $\Norm{A}_2 \leq 1$. We get then
    \begin{align*}
        \Norm{P^{1}(A) - \op{PROD}_\varepsilon^2(A)}_2 & = \left\lVert \frac{A+\Eye_n}{2}
        \left(\left(\frac{A}{2}\right)^2 + \frac{\Eye_n}{4}\right) \right . \\
        & \left . - R \left ( \Phi^{\op{STR}_{n, \varrho}}_{\frac{\varepsilon}{4 n},1} \right )\left(\op{SQR}^1_\varepsilon(A) +\frac{\Eye_n}{4} \bigg| \frac{A+\Eye_n}{2} \right) \right \rVert_2 \\
        & \leq \Norm{\left ( \left(\frac{A}{2}\right)^2 - \op{SQR}_\varepsilon^1(A) \right)\left(\frac{A+\Eye_n}{2}\right)}_2 \\
        & + \left \lVert \left ( \op{SQR}_\varepsilon^1(A) + 2^{-2} \Eye_n \right ) \left ( \frac{A+\Eye_n}{2} \right )  \right . 
        \\ & \left . - R \left ( \Phi^{\op{STR}_{n, \varrho}}_{\frac{\varepsilon}{4 n}, 1} \right )\left(\op{SQR}^1_\varepsilon(A) + 2^{-2} \Eye_n \bigg| \frac{A+\Eye_n}{2} \right) \right \rVert_2 \\
        & \leq \frac{\varepsilon}{4} + \frac{\varepsilon}{4} \leq \varepsilon,
    \end{align*}
    where we have used that
    \[
    \Norm{\frac{A+\Eye_n}{2}}_2\leq 1,
    \]
    that
    \[
    \Norm{\op{SQR}^1_\varepsilon(A) + 2^{-2} \Eye_n}_2 \leq 1,
    \]
    as well as \eqref{eq:sqrepsdefalt},  \eqref{eq: SQR N times MNN N times square} and \eqref{eq: infinity and 2 norm equivalence constants}.
    
    Let us suppose that the result \eqref{eq: MNN AUX approximates Neu except by constant} is true for $i \geq 2$ and prove it for $i + 1$. Using \eqref{eq: PN norm 2 upper bound} and the inductive hypothesis \eqref{eq: MNN AUX approximates Neu except by constant}, we get that
    \[
    \Norm{\op{PROD}^i_\varepsilon(A)}_2 \leq \varepsilon + \Norm{P^i(A)}_2 \leq \varepsilon + \frac{1}{2} < 1.
    \]
    Consequently, using \eqref{eq: infinity and 2 norm equivalence constants} and \eqref{eq: SQR norm 2 upper bound} we get:
    \begin{equation*}
        \left \lVert \left (\op{SQR}_\varepsilon^i(A) + 2^{-2^i} \Eye_n \right )\op{PROD}_\varepsilon^i(A) - R \left ( \Phi_{\frac{\varepsilon}{4 n}, 1}^{\op{STR}_{n, \varrho}} \right ) \left ( \op{SQR}_\varepsilon^i(A) + 2^{-2^i} \Eye_n \middle | \op{PROD}_\varepsilon^i(A) \right ) \right \rVert_2 \leq \frac{\varepsilon}{4}.
    \end{equation*}
    We conclude then
    \begin{align*}
        & \Norm{P^i(A) - \op{PROD}_\varepsilon^{i + 1}(A)}_2 = \left \lVert \left ( 2^{-2^i} \Eye_n + \left(\frac{A}{2}\right)^{2^i} \right ) P^{i - 1}(A) \right . \\
        & \left .- R \left ( \Phi_{\frac{\varepsilon}{4 n}, 1}^{\op{STR}_{n, \varrho}} \right ) \left ( \op{SQR}_\varepsilon^i(A) + 2^{-2^i} \Eye_n \middle | \op{PROD}_\varepsilon^i(A) \right ) \right \rVert_2 \\
        & \leq \Norm{\left ( 2^{-2^i} \Eye_n + \left(\frac{A}{2}\right)^{2^i} \right ) P^{i - 1}(A) - \left ( \op{SQR}_\varepsilon^i(A) + 2^{-2^i} \Eye_n \right ) \op{PROD}_\varepsilon^i(A)}_2 + \frac{\varepsilon}{4} \\
        & \leq \Norm{\left ( \left(\frac{A}{2}\right)^{2^i} - \op{SQR}_\varepsilon^i(A) \right ) P^{i- 1}(A)}_2 \\
        & + \Norm{\left ( 2^{-2^i} \Eye_n + \op{SQR}^i_\varepsilon(A) \right ) \left ( P^{i - 1}(A) - \op{PROD}_\varepsilon^i(A) \right )}_2 + \frac{\varepsilon}{4} \\
        & \leq \varepsilon \left ( \Norm{P^{i - 1}(A)}_2 + \Norm{2^{-2^i} \Eye_n + \op{SQR}_\varepsilon^i(A)}_2 + \frac{1}{4} \right ) \\
        & \leq \varepsilon \left ( 2^{i - 2^i} + 2^{-2^i} + \varepsilon + 2^{-2^i} + \frac{1}{4} \right ) \leq \varepsilon,
    \end{align*}
    which implies \eqref{eq: MNN AUX approximates Neu except by constant}. This also implies \eqref{eq: NEU MNN Neumann series} because for all $A \in \R^{n \times n}$:
    \begin{equation*}
     R \left ( \Phi^{\op{NEU}^N_{n, \varrho}}_\varepsilon \right )(A) 
     = 2^{2^N - 1} R \left ( \Phi_{2^{1 - 2^N} \frac{\varepsilon}{4n}, 1}^{\op{STR}_{n, \varrho}} \right ) \left ( \op{SQR}^{N - 1}_{2^{1 - 2^N} \varepsilon}(A) + 2^{-2^{N - 1}} \Eye_n \middle | \op{PROD}_{2^{1 - 2^N} \varepsilon}^{N - 1} (A) \right ) 
    = \op{PROD}_\varepsilon^N(A).
    \end{equation*}
    where we used the definitions of $\Phi_\varepsilon^{\op{NEU}^N_{n, \varrho}}$, $\op{PROD}_{2^{- 2^N} \varepsilon}^N$, and $\Phi^{\op{FLIP}}_{n, N - 1}$.

    \textbf{Step 5:} Proof of \eqref{eq: NEU MNN weights} for $N\geq2$.
        We can rapidly deduce with \eqref{eq: number of weights and layers of auxiliary MNN upper bounds} that
    \begin{align*}
        M \left ( \Phi_\varepsilon^{\op{NEU}^N_{n, \varrho}} \right ) & = M \left ( \Phi_{2^{1-2^N}\frac{\varepsilon}{4 n}, 1}^{\op{\op{STR}_{n, \varrho}}} \right ) + M \left ( \Phi^{\op{FLIP}}_{n, N - 1} \right ) \\
        & + \sum_{k = 2}^{N - 1} \left [ M \left ( \cali{P} \left ( \Phi_{2^{1-2^N}\frac{\varepsilon}{4 n}, 1}^{\op{\op{STR}_{n, \varrho}}}, \Phi_{2^{1-2^N}\frac{\varepsilon}{4 n}, 1}^{\op{\op{STR}_{n, \varrho}}} \right ) \right ) + M \left ( \Phi_{n, k - 1}^{\op{MIX}} \right ) \right ] \\
        & + M \left ( \cali{P} \left ( \Phi_{2^{1-2^N}\frac{\varepsilon}{4 n}, 1}^{\op{\op{STR}_{n, \varrho}}}, \Phi_{n, L \left ( \Phi_{2^{1-2^N}\frac{\varepsilon}{4 n}, 1}^{\op{\op{STR}_{n, \varrho}}} \right )}^{\op{FILL}} \right ) \right ) + M \left ( \Phi^{\op{DUP}}_n \right ) \\
        & = 2(N - 1) M \left ( \Phi_{2^{-2^N}\frac{\varepsilon}{2 n}, 1}^{\op{\op{STR}_{n, \varrho}}} \right ) + 2 n^2 + n + (N - 2)(4 n^2 + n) \\
        & + n^2 L \left ( \Phi_{2^{-2^N}\frac{\varepsilon}{2 n}, 1}^{\op{\op{STR}_{n, \varrho}}} \right ) + n + 3 n^2.
        \end{align*}
        Thus, using \eqref{eq: MNN STR arbitrary square weights} and \eqref{eq: MNN STR arbitrary square layers}, we obtain that:  
        \begin{align*}
        M \left ( \Phi_\varepsilon^{\op{NEU}^N_{n, \varrho}} \right )
        & \leq 2(N - 1) \left [ 7 n^{\log_2 7} \left ( M \left ( \Pi_{2^{1 - 2^N} \frac{\varepsilon}{16 n^3}, 2n}^\varrho \right ) + 12 \right ) - 9 n^2 \right ] \\
        & + n^2 \left [ 4 N - 3 + L \left ( \Pi_{2^{1 - 2^N} \frac{\varepsilon}{16 n^3}, 2n}^\varrho \right ) + 2 (\log_2 n + 2) \right ] + n N \\
        & = 14 n^{\log_2 7} (N - 1) \left ( M \left ( \Pi_{2^{ - 2^N} \frac{\varepsilon}{8 n^3}, 2n}^\varrho \right ) + 12 \right ) \\
        & + n^2 \left ( L \left ( \Pi_{2^{ - 2^N} \frac{\varepsilon}{8 n^3}, 2n}^\varrho \right ) - 14 N + 19 + 2 \log_2 n\right ) + n N.
    \end{align*}
    
    \textbf{Step 6:} Proof of \eqref{eq: NEU MNN layers} for $N\geq2$.
    For the number of layers \eqref{eq: NEU MNN layers}, we can similarly to Step 5 get that
    \begin{equation*}
        L \left ( \Phi_\varepsilon^{\op{NEU}^N_{n, \varrho}} \right )  = N \left ( L \left ( \Phi_{2^{1-2^N}\frac{\varepsilon}{4 n}, 1}^{\op{\op{STR}_{n, \varrho}}} \right ) + 1 \right ) 
        \leq N \left [ 2\log_2 n + 5 + L \left ( \Pi_{2^{- 2^N} \frac{\varepsilon}{8 n^3}, 2 n}^\varrho \right ) \right ],
    \end{equation*}
    using again \eqref{eq: number of weights and layers of auxiliary MNN upper bounds} and \eqref{eq: MNN STR arbitrary square layers}.
\end{proof}

Let us now prove Theorem \ref{th: INV MNN}:
\begin{proof}
    We first focus on proving \eqref{eq: INV MNN inverse}. By Remark \ref{rm: inverse as Neumann sum} with $P=\alpha\Eye_n$, we know that the identity
    \[
    A^{-1} = \alpha \sum_{k = 0}^{\infty} (\Eye_n - \alpha A)^k
    \]
    holds for all $A \in \R^{n \times n}$ such that $\Norm{\Eye_n - \alpha A}_2 < 1$. By \eqref{eq: Neumann partial sum error}, we have that
    \begin{equation} \label{est:sumalphdelt}
        \Norm{A^{-1} - \alpha \sum_{k = 0}^{2^N-1} (\Eye_n - \alpha A)^k}_2  = \alpha \Norm{(\Eye_n-(\Eye_n-\alpha A))^{-1} - \sum_{k = 0}^{2^N-1} (\Eye_n - \alpha A)^k} 
\leq \alpha \frac{\delta^{2^N}}{1 - \delta}
    \end{equation}
    when 
    \[
    \Norm{\Eye_n - \alpha A}_2 \leq \delta \in (0, 1).
    \]
    Our plan is then to use the MNN in Lemma \ref{th: NEU MNN} to approximate matrix inversion. It is easy to check that for $N(\varepsilon,\delta)$ given by \eqref{def:nepsdel}
    we get the inequality:
    \begin{equation}\label{eq:deltNeps}
        \frac{\delta^{2^{N(\varepsilon,\delta)}}}{1 - \delta} \leq \varepsilon.
    \end{equation}
    Consider the shallow MNN $\Phi^{\op{IN}}_{n, \alpha}$ that satisfies:
    \[
    R \left ( \Phi^{\op{IN}}_{n,\alpha} \right ) (A) = \Eye_n - \alpha A
    \]
    for all $A \in \R^{n \times n}$ and 
    \begin{equation}\label{eq:wphiN}
        M \left ( \Phi^{\op{IN}}_n \right ) = n + n^2.
    \end{equation}
    If we define
    \[
    ((\LL^1, B^1, \alpha^1), \dots, (\LL^L, B^L, \alpha^L)) \coloneq \Phi_{\min \set{\frac{\varepsilon}{2 \alpha}, \frac{1}{8}}}^{\op{NEU}^{N\left ( \frac{\varepsilon}{2 \alpha}, \delta \right )}_{n, \varrho}},
    \]
    then our target MNN is 
    \[
    \Phi_{\varepsilon, \delta}^{\op{INV}^\alpha_{n, \varrho}} \coloneq ((\LL^1, B^1, \alpha^1), \dots, (\alpha \LL^L, \alpha B^L, \alpha^L)) \odot \Phi_{n, \alpha}^{\op{IN}}.
    \]
    Indeed, we get for all $A \in \R^{n \times n}$ such that $\Norm{\Eye_n - \alpha A}_2 \leq \delta$:
    \begin{align*}
        & \Norm{A^{-1} - R \left ( \Phi_{\varepsilon, \delta}^{\op{INV}^\alpha_{n, \varrho}} \right )(A)}_2 \\
        & \leq \Norm{A^{-1} - \alpha \sum^{2^{N\left ( \frac{\varepsilon}{2 \alpha}, \delta \right )} - 1}_{k = 0} (\Eye_n - \alpha A)^k}_2 \\
        & + \Norm{\alpha \sum^{2^{N\left ( \frac{\varepsilon}{2 \alpha}, \delta \right )} - 1}_{k = 0} (\Eye_n - \alpha A)^k - \alpha R \left ( \Phi_{\min \set{\frac{\varepsilon}{2 \alpha}, \frac{1}{8}}}^{\op{NEU}^{N\left ( \frac{\varepsilon}{2 \alpha}, \delta \right )}_{n, \varrho}} \right )(\Eye_n - \alpha A)}_2 \\
        & \leq \alpha \frac{\delta^{2^{N\left ( \frac{\varepsilon}{2 \alpha}, \delta \right )}}}{1 - \delta} + \alpha \min \set{\frac{\varepsilon}{2 \alpha}, \frac{1}{8}} \leq \varepsilon.
    \end{align*}
    In the second estimate, we have used \eqref{est:sumalphdelt} and \eqref{eq: NEU MNN Neumann series}, and in the third one \eqref{eq:deltNeps}. 

    We follow by proving the upper bound in the number of weights. For that, we consider that:
    \begin{equation}\label{eq:MInvMNeu}
    M \left ( \Phi_{\varepsilon, \delta}^{\op{INV}^\alpha_{n, \varrho}} \right )  = M \left ( \Phi_{\min \set{\frac{\varepsilon}{2 \alpha}, \frac{1}{8}}}^{\op{NEU}^{N\left ( \frac{\varepsilon}{2 \alpha}, \delta \right )}_{n, \varrho}} \right ) + M \left ( \Phi^{\op{IN}}_{n, \alpha} \right ).
    \end{equation}
    In fact, if $N\left ( \frac{\varepsilon}{2 \alpha}, \delta \right ) \geq 2$ we obtain \eqref{eq: INV MNN weights} from \eqref{eq:MInvMNeu}, \eqref{eq: NEU MNN weights} and \eqref{eq:wphiN}.  Similarly, if $N\left ( \frac{\varepsilon}{2 \alpha}, \delta \right ) = 1$,
    we obtain \eqref{eq: INV MNN weights N equals 1} 
    from \eqref{eq:MInvMNeu}, \eqref{eq: NEU MNN N equals 1 weights} and \eqref{eq:wphiN}.

    We finish by proving the upper bound in the number of layers, that is, by proving \eqref{eq: INV MNN layers} and \eqref{eq: INV MNN layers N equals 1}. For that, we only need to consider:
    \[L \left ( \Phi_{\varepsilon, \delta}^{\op{INV}^\alpha_{n, \varrho}} \right )  = L \left ( \Phi_{\min \set{\frac{\varepsilon}{2 \alpha}, \frac{1}{8}}}^{\op{NEU}^{N\left ( \frac{\varepsilon}{2 \alpha}, \delta \right )}_{n, \varrho}} \right ) +1,\]
    and the estimates \eqref{eq: NEU MNN layers} and \eqref{eq: NEU MNN N equals 1 layers} respectively.
\end{proof}

\bibliographystyle{abbrv} 
\bibliography{NNStrassen}


\end{document}